\documentclass[options]{article}
\pdfoutput=1
\usepackage{amsmath}
\usepackage{amsfonts}
\usepackage{amssymb}
\usepackage{amsthm}
\usepackage{esint}
\usepackage{mathtools}
\usepackage{titling}
\usepackage{authblk}
\usepackage{commath}
\usepackage{xcolor}
\usepackage{graphicx}
\usepackage{graphbox}
\usepackage[toc,title,page]{appendix}
\usepackage{subcaption}
\usepackage{bbm}
\usepackage{hyperref}
\usepackage{color,soul}
\usepackage{bm}
\usepackage{mathrsfs}

\usepackage{enumitem}
\usepackage[margin=1in]{geometry}

\newtheorem{theorem}{Theorem}
\newtheorem{corollary}{Corollary}
\newtheorem{lemma}{Lemma}
\newtheorem{remark}{Remark}

\newtheorem{proposition}{Proposition}
\newtheorem*{remark*}{Remark}

\usepackage[backend=biber,citestyle=numeric,bibstyle=musuos,sorting=nyt,dashed=false]{biblatex}
\makeatletter
\RequireBibliographyStyle{numeric}
\makeatother

\newbibmacro*{cite:labelyear+extrayear}{%
  \iffieldundef{labelyear}
    {}
    {\printtext[bibhyperref]{%
       \printfield{labelyear}%
       \printfield{extrayear}}}}

\title{Global Existence for Nonlocal Quasilinear Diffusion Systems \\in Non-Isotropic Non-Divergence Form}
\author{Catharine W.K. Lo\thanks{CMAFcIO -- Departamento de Matem\'atica, Faculdade de Ci\^encias, Universidade de Lisboa P-1749-016 Lisboa, Portugal\\ Email address: cwklo@fc.ul.pt} \, and Jos\'e Francisco Rodrigues\thanks{CMAFcIO -- Departamento de Matem\'atica, Faculdade de Ci\^encias, Universidade de Lisboa P-1749-016 Lisboa, Portugal\\ Email address: jfrodrigues@ciencias.ulisboa.pt}}
\date{}

\begin{document}
\maketitle

\begin{abstract}
    Consider the quasilinear diffusion problem \begin{equation*}\begin{cases}\bm{u}'+\Pi(t,x,\bm{u},\Sigma \bm{u})\mathbb{A}\bm{u}=\bm{f}(t,x,\bm{u},\Sigma \bm{u})&\text{ in }]0,T[\times\Omega,\\\bm{u}=\bm{0}&\text{ in }]0,T[\times\Omega^c,\\\bm{u}(0,\cdot)=\bm{u}_0(\cdot)&\text{ in }\Omega\end{cases}\end{equation*} for an open set $\Omega\subset\mathbb{R}^n$, $\bm{u}_0\in \mathbf{H}^s_0(\Omega):=[H^s_0(\Omega)]^m$ and any $T\in]0,\infty[$, where $\Sigma \bm{u}\in \mathbb{R}^q$ for $0<q\leq m\times n$ represents fractional or nonlocal derivatives with order $\sigma$ with $\sigma<2s$ for all $0<s\leq1$, including the classical gradient and derivatives of order greater than 1. We show global existence results for various quasilinear diffusion systems in non-divergence form, for different linear operators $\mathbb{A}$, including local elliptic systems, anisotropic fractional equations and systems, and anisotropic nonlocal operators, of the following type \[(\mathbb{A}\bm{u})^i=-\sum _{\alpha,\beta,j} \partial_\alpha(A^{\alpha\beta}_{ij}\partial_\beta u^j),\quad \mathbb{A}u=- D^s(A(x)D^su),\quad\text{ and }\quad (\mathbb{A}\bm{u})^i=\int_{\mathbb{R}^n}A_{ij}(x,y)\frac{u^j(x)-u^j(y)}{|x-y|^{n+2s}}\,dy,\] for coercive, invertible matrices $\Pi$ and suitable vectorial functions $\bm{f}$.
\end{abstract}

\section{Introduction}

Consider the quasilinear diffusion problem for $\bm{u}=(u^1,\dots,u^m)=\bm{u}(t,x)$ \begin{equation}\label{MainProblemGlobalQuasilinear}\begin{cases}\bm{u}'+\Pi(t,x,\bm{u},\Sigma \bm{u})\mathbb{A}\bm{u}=\bm{f}(t,x,\bm{u},\Sigma \bm{u})&\text{ in }]0,T[\times\Omega,\\\bm{u}=\bm{0}&\text{ in }]0,T[\times\Omega^c,\\\bm{u}(0,\cdot)=\bm{u}_0(\cdot)&\text{ in }\Omega\end{cases}\end{equation} for an open (bounded or unbounded) set $\Omega\subset\mathbb{R}^n$, $\bm{u}_0\in \mathbf{H}^s_0(\Omega):=[H^s_0(\Omega)]^m$ and any $T\in]0,\infty[$, where $\Sigma \bm{u}\in \mathbb{R}^q$ for $0<q\leq m\times n$ represents fractional or nonlocal derivatives in the form $D^\sigma \bm{u}$ or $\mathcal{D}^\sigma \bm{u}$ for $\sigma< 2s$, $0<s\leq1$, $\sigma$ possibly equal or greater than 1 including the classical gradient, $\mathbb{A}$ is symmetric time-independent local or nonlocal operator, which is bounded and $\mathbf{L}^2(\Omega)=[L^2(\Omega)]^m$-coercive, i.e. \begin{equation}\label{LEllipBdd}\begin{split}\langle\mathbb{A}\bm{u},\bm{v}\rangle\leq a^*\norm{\bm{u}}_{\mathbf{H}^s_0(\Omega)}\norm{\bm{v}}_{\mathbf{H}^s_0(\Omega)}\text{ for some }a^*>0,\,\forall \bm{u},\bm{v}\in \mathbf{H}^s_0(\Omega),\text{ and}\\\langle \mathbb{A}\bm{u},\bm{u}\rangle+\mu\norm{\bm{u}}_{\mathbf{L}^2(\Omega)}^2\geq a_*\norm{\bm{u}}_{\mathbf{H}^s_0(\Omega)}^2\text{ for some }\mu\geq0,\,a_*>0,\,\forall \bm{u}\in \mathbf{H}^s_0(\Omega),\end{split}\end{equation} for the classical Sobolev space $\mathbf{H}^s_0(\Omega), 0<s\leq1$, so that $\mathbb{A}:\mathbf{H}^s_0(\Omega)\to \mathbf{H}^{-s}(\Omega)$ is linear and continuous.
Suppose also that $\bm{f}:]0,T[\times\Omega\times\mathbb{R}^m\times\mathbb{R}^q\to\mathbb{R}^m$ is measurable such that it is continuous with respect to $\bm{u}$ and $\Sigma \bm{u}$ for almost every $(t,x)$ and satisfies a linear growth condition with respect to the last variable, and $\Pi:]0,T[\times\Omega\times\mathbb{R}^m\times\mathbb{R}^q\to\mathbb{R}^{m\times m}$ is a measurable, coercive, invertible matrix such that it is continuous with respect to $\bm{u}$ and $\Sigma \bm{u}$ for almost every $(t,x)$ and \begin{equation}\label{GammaNonDivCond}\underline{\gamma}|\xi|^2\leq\Pi\xi\cdot\xi\quad\text{ and }\quad\Pi\xi\cdot\xi^*\leq\bar{\gamma}|\xi||\xi^*|\quad\text{ for all }\xi,\xi^*\in\mathbb{R}^m\end{equation} for all $\bm{u}$ and $\Sigma\bm{u}$ and almost all $(t,x)$ with  $0<\underline{\gamma}\leq\bar{\gamma}$.

The main purpose of this work is to prove the existence of a solution $\bm{u}$ to Problem \eqref{MainProblemGlobalQuasilinear} in the space \[H^1(0,T;\mathbf{L}^2(\Omega))\cap L^2(0,T;\mathbf{L}^2_{\mathbb{A}})\cap C([0,T];\mathbf{H}^s_0(\Omega)).\] Here $\mathbf{L}^2_{\mathbb{A}}=D(\mathbb{A})$ is the domain of the operator $\mathbb{A}$, associated with homogeneous Dirichlet boundary condition when $\mathbb{A}u\in \mathbf{L}^2(\Omega)$, given by \[\mathbf{L}^2_{\mathbb{A}}=D(\mathbb{A}):=\{\bm{u}\in \mathbf{H}^s_0(\Omega):\mathbb{A}\bm{u}\in \mathbf{L}^2(\Omega)\},\] as $\mathbb{A}$ may be regarded as an operator in the classical framework $\mathbf{H}^\sigma_0(\Omega)\subset\mathbf{L}^2(\Omega)\subset\mathbf{H}^{-s}(\Omega)$.
Then, because the operator $\mathbb{A}$ is closed, the space $\mathbf{L}^2_{\mathbb{A}}$ is a Hilbert space when equipped with the graph norm for any $\Omega\subseteq\mathbb{R}^n$. Subsequently, the Bochner space $L^2(0,T;\mathbf{L}^2_{\mathbb{A}})$ is also a Hilbert space.

Note that here $\bm{f}_{\bm{u}}(t,x)=\bm{f}(t,x,\bm{u}(t,x),\Sigma\bm{u}(t,x))$ and $\Pi_{\bm{u}}(t,x)=\Pi(t,x,\bm{u}(t,x),\Sigma\bm{u}(t,x))$ are functions in $L^2(0,T;\mathbf{L}^2(\Omega))$ and $L^\infty(]0,T[\times\Omega)$ respectively.
Problem \ref{MainProblemGlobalQuasilinear} generalises the quasilinear equation defined with the classical gradient in \cite{ArendtChill} to systems of equations with more general derivatives. 

Following \cite{Comi1} and \cite{BellidoCuetoMoraCorral2020PiolaVector} (see also \cite{SS1+2}), for $0<s<1$, the Riesz \emph{fractional gradient} $D^s\bm{u}$ may be defined component-wise in integral form for vectors $\bm{u}=(u^1,u^2,\dots,u^m)\in\mathbf{H}^s_0(\Omega)$, respectively, by  \begin{equation}\label{FracDervs}D^s_iu^j(x):=c_{n,s}\int_{\mathbb{R}^n}\frac{u^j(x)-u^j(y)}{|x-y|^{n+s}}\frac{x_i-y_i}{|x-y|}\,dy,\qquad i=1,\dots,n,\quad j=1,\dots,m\end{equation} where $c_{n,s}=2^s\pi^{-\frac{n}{2}}\frac{\Pi\left(\frac{n+s+1}{2}\right)}{\Pi\left(\frac{1-s}{2}\right)}$ is given in terms of the Gamma-function, and $\bm{u}$ is extended by $\bm{0}$ outside $\Omega$, supposed to satisfy the extension property if $s>1/2$, so that we can assume that the extension of $\bm{u}$ is in $\mathbf{H}^s(\mathbb{R}^n)$ whenever $\bm{u}\in \mathbf{H}^s_0(\Omega)$. 
Similarly, 
the \emph{nonlocal gradient} $\mathcal{D}^s\bm{u}$ is defined component-wise, as in \cite{DuGunzburgerLehoucqZhouNonlocalVectorCalculus}, by \begin{equation}\label{NonlocalDervs}\mathcal{D}^su^j(x,y):=\frac{u^j(x)-u^j(y)}{|x-y|^{\frac{n}{2}+s}},\qquad j=1,\dots,m.\end{equation}
For $1<s<2$, we consider only the fractional gradients, defined by 
\begin{equation}\label{HigherOrderFracDervs}D^s_iu^j=D^{s-1}_i(\partial_i u^j),\end{equation} for the classical partial derivative $\partial_i=\frac{\partial}{\partial x_i}$. This is possible by making use of the semigroup property of the Riesz potentials and the property of the distributional Riesz fractional gradients which can be given through the convolution with it (Theorem 1.2 of \cite{SS1+2}(I)). Note that it is not possible to define a higher order nonlocal gradient, since $\mathcal{D}^su^j\not\in L^2(\mathbb{R}^n\times\mathbb{R}^n)$ for $s>1$.

$\mathbb{A}$ may be given by linear combinations of the classical gradient $\partial$, $D^s$ or $\mathcal{D}^s$, as long as it is bounded and $\mathbf{L}^2(\Omega)$-coercive, as in the following examples.  When $s=1$, this includes the local operator given by \begin{equation}\label{ClassicalDeriv}\langle\mathbb{A}\bm{u},\bm{v}\rangle=\langle\mathbb{L}\bm{u},\bm{v}\rangle=\sum _{\alpha,\beta,i,j} \int_\Omega A^{\alpha\beta}_{ij}\partial_\beta u^j\cdot \partial_\alpha v^i\end{equation} with a bounded, coercive tensor $A=(A^{\alpha\beta}_{ij}(x))$, symmetric in $\alpha$ and $\beta$, where $\langle\mathbb{A}\bm{u},\bm{v}\rangle$ is understood as the duality between $\mathbf{H}^{-1}(\Omega)$ and $\mathbf{H}^1_0(\Omega)$. 
The sum here is taken between 1 to $n$ for $\alpha$ and $\beta$, and between 1 to $m$ for $i$ and $j$. 
$\mathbb{A}$ may also be the anisotropic fractional operator \begin{equation}\label{FracOp}\langle\mathbb{A}\bm{u},\bm{v}\rangle=\langle\tilde{\mathcal{L}}^s_A\bm{u},\bm{v}\rangle=\sum_{\alpha,\beta,i,j}\int_{\mathbb{R}^n} A_{ij}^{\alpha\beta}D^s_\beta u^j\cdot D^s_\alpha v^i\end{equation} for $s\leq1$, where $D^s_\alpha$ coincides with $\partial_\alpha$ in the classical case of $s=1$, where $\langle\mathbb{A}\bm{u},\bm{v}\rangle$ is understood as the duality between $\mathbf{H}^{-s}(\Omega)$ and $\mathbf{H}^s_0(\Omega)$. 

We can also consider the anisotropic nonlocal operator $\mathbb{A}:\mathbf{H}^s_0(\Omega)\to\mathbf{H}^{-s}(\Omega)$ \begin{equation}\label{NonlocalOp}\mathbb{A}\bm{u}=\mathbb{L}^s_A\bm{u}=P.V.\int_{\mathbb{R}^n}A(x,y)\frac{\bm{u}(x)-\bm{u}(y)}{|x-y|^{n+2s}}\,dy\end{equation} defined for a symmetric, bounded, coercive matrix kernel $A=A_{ij}(x,y)$, i.e. for almost all $(x,y)$ in $\mathbb{R}^n\times\mathbb{R}^n$,
\[a_*|\xi|^2\leq A\xi\cdot\xi\leq a^*|\xi|^2\quad\text{ for all }\xi\in\mathbb{R}^m,\]
for $s<1$, so that, for all $\bm{u},\bm{v}\in \mathbf{H}^s_0(\Omega)$,
\[\langle\mathbb{L}^s_A\bm{u},\bm{v}\rangle=\int_{\mathbb{R}^n}\int_{\mathbb{R}^n}A(x,y)\frac{\bm{u}(x)-\bm{u}(y)}{|x-y|^{\frac{n}{2}+s}}\cdot\frac{\bm{v}(x)-\bm{v}(y)}{|x-y|^{\frac{n}{2}+s}}\,dy\,dx\leq a^*\norm{\bm{u}}_{\mathbf{H}^s_0(\Omega)}\norm{\bm{v}}_{\mathbf{H}^s_0(\Omega)},\] and
\[\langle \mathbb{L}^s_A\bm{u},\bm{u}\rangle\geq a_*\norm{\bm{u}}_{\mathbf{H}^s_0(\Omega)}^2.\]

The fractional Sobolev spaces $H^s(\mathbb{R}^n)$ for all real $s$ are defined by
\[H^s(\mathbb{R}^n)=\{u\in L^2(\mathbb{R}^n):\{\xi\mapsto(1+|\xi|^2)^{s/2}\hat{u}(\xi)\}\in L^2(\mathbb{R}^n)\},\] with norm \[\norm{u}_{H^s(\mathbb{R}^n)}=\norm{(1+|\xi|^2)^{s/2}\hat{u}}_{L^2(\mathbb{R}^n)},\] where $\hat{u}$ is the Fourier transform of $u$. For $0<s<1$, this norm is well known to be equivalent to \begin{equation}\label{HsGraphNorm}\norm{u}_{H^s(\mathbb{R}^n)}^2=\norm{u}_{L^2(\mathbb{R}^n)}^2+\int_{\mathbb{R}^n}\int_{\mathbb{R}^n}\frac{|u(x)-u(y)|^2}{|x-y|^{n+2s}}\,dx\,dy=:\norm{u}_{L^2(\mathbb{R}^n)}^2+[u]_{H^s(\mathbb{R}^n)}^2.\end{equation} 
Then, if $\Omega$ has Lipschitz boundary, hence satisfying the extension property,  $H^s(\Omega)$ coincides with the space of restrictions to $\Omega$ of the elements of $H^s(\mathbb{R}^n)$ as in \cite{LionsMagenesBook1} and \cite{Demengel}, with norm \begin{equation}\label{HsRestrictedNorm}\norm{u}_{H^s(\Omega)}=\inf_{U=u\text{ a.e.}\,\Omega\text{ }}\,\norm{U}_{H^s(\mathbb{R}^n)}.\end{equation} On the other hand, as it was shown in \cite{SS1+2} and \cite{FracObsRiesz}, the $H^s(\Omega)$-norm given by \eqref{HsGraphNorm} is in fact equal to \begin{equation}\label{HsEquivNorm}\norm{u}_{H^s(\Omega)}^2=\norm{u}_{L^2(\Omega)}^2+\frac{2}{c_{n,s}^2}\int_{\mathbb{R}^n}|D^su|^2=\norm{u}_{L^2(\Omega)}^2+\frac{2}{c_{n,s}^2}\norm{D^su}_{L^2(\mathbb{R}^n)}^2,\end{equation} where $u$ is extended by zero outside $\Omega$.

Here, the subspace $H^s_0(\Omega)$ is the usual Sobolev space, for $0<s\leq1$, given by the closure of $C_c^\infty(\Omega)$ in $H^s(\Omega)$ for general open sets $\Omega\subset\mathbb{R}^n$, as in \cite{LionsMagenesBook1}, and $H^{-s}(\Omega)$ its dual. Since $C_c^\infty(\Omega)$ is dense in $H^s(\Omega)$ if and only if $s\leq\frac{1}{2}$, in this case, $H^s_0(\Omega)=H^s(\Omega)$. Otherwise, if $s>\frac{1}{2}$, $H^s_0(\Omega)$ is strictly contained in $H^s(\Omega)$. On the other hand, as in \cite{Demengel}, for bounded sets with Lipschitz boundary, $\mathcal{O}\subset\mathbb{R}^n$, $C_c^\infty(\bar{\mathcal{O}})$ is dense in $H^s(\mathcal{O})$ for all $s\geq0$.

This can be further extended for $s>1$, by an abuse of notation, by defining $H^s_0(\Omega)$ to be the space \[H^s_0(\Omega):=\{u\in H^s(\mathbb{R}^n):\text{ supp }u\subset\bar{\Omega}\}.\]

Consider the \emph{maximal regularity space} \[MR:=H^1(0,T;\mathbf{L}^2(\Omega))\cap L^2(0,T;\mathbf{L}^2_{\mathbb{A}}),\] equipped with the norm for $0<s\leq1$ \begin{equation}\label{MRnorm}\norm{\bm{u}}_{MR}^2:=\int_0^T\norm{ \bm{u}'(t)}_{\mathbf{L}^2(\Omega)}^2+\int_0^T\norm{\bm{u}(t)}_{\mathbf{H}^s_0(\Omega)}^2+\int_0^T\norm{\mathbb{A}\bm{u}(t)}_{\mathbf{L}^2(\mathbb{R}^n)}^2,\end{equation} so that the linear inhomogeneous problem \[ \bm{u}'(t)+\mathbb{A}\bm{u}(t)=\bm{f}(t)\quad\text{ for a.e. }t\in ]0,T[,\quad \bm{u}(0)=\bm{0}\] is well-defined with a source term $\bm{f}\in L^2(0,T;\mathbf{L}^2(\Omega))$.

Classically, parabolic quasilinear systems in non-divergence form have frequently been considered (see \cite{Friedman1958ParabolicSystems}, \cite{Eidelman1964book} \cite{LadyzhenskayaSolonnikovUraltsevaBook}, \cite{IvockinaOskolkov1970ParaSystems}, \cite{AmannParabolicV1},  \cite{PucciSerrin1996AsymptoticsNonlinearParabolicSystems}, \cite{HallerDintelmannHeckHieber2006LpLqParaSystemsNonDivVMO}, \cite{ArkhipovaStara2017ParaSystems} and their references), with multiple physical, chemical and biological applications such as in reaction-diffusion systems (see, for example, \cite{Morgan1989GlobalExistenceSemilinearParabolicSystems}), phase-field models (see, for example, \cite{PhaseFieldFracture}) and population models (see, for instance, \cite{LeungNonlinearPDEBookPart2} and \cite{BendahmaneLanglais2010JEvolEqCrossDiffLinearizedEq1.10}). Parabolic equations have also been extended to the case of nonlocal reaction-diffusion (see, for instance, \cite{NonlocalDiffusionEg1} and \cite{BucurValdinociNonlocalDiffusionLectures}).

In Section \ref{sect:MRQuasiLinear}, we will first consider the linear problem, extending the approach of \cite{ArendtChill} to systems by introducing a suitable time-dependent matrix $\Upsilon$, thereby obtaining the solution to the non-autonomous linear problem given with the well-known maximal regularity and exemplifying with three linear systems of the above type, which may have additional regularity in bounded Lipschitz domains. Next, we will extend the result by a fixed point argument to obtain the result for the quasilinear problem. Limited by the currently known regularity of the Dirichlet problems associated to the operator $\mathbb{A}$, in Section \ref{sect:MRQuasiSigmaSmall}, we obtain the existence of a solution for the global quasilinear nondivergent systems for the general operators $\mathbb{A}$ satisfying \eqref{LEllipBdd}, first for $\sigma< s$  and also for particular operators satisfying additional regularity properties, up to and including $\sigma=s\leq1$. This extends the nonlocal vectorial problem with no source function considered in \cite{LaasriMugnolo2020MMAS}, as well as the vectorial semilinear case in \cite{Amann1985GlobalExistenceSemilinearParabolicSystems}, \cite{Morgan1989GlobalExistenceSemilinearParabolicSystems} and \cite{ArendtDierSystems}. This also generalises \cite{ArendtChill} to systems of the form \eqref{MainProblemGlobalQuasilinear} defined in a bounded or unbounded open set $\Omega\subset\mathbb{R}^n$, for more general derivatives that can take any positive order less than $s$, which is an improvement even in the classical case of $s=1$. 

This result is then further generalised to larger $s<\sigma<2s$ in the case of $\Omega$ bounded with Lipschitz boundary, making use of known regularity results for vectorial local and nonlocal operators in Section \ref{sect:MRQuasiSigmaBig}, in particular generalising to quasilinear diffusion systems the classical scalar Dirichlet case of \cite{ArendtChill}. As a result, we can also consider quasilinear diffusion equations and systems with derivatives of order $\sigma>s$ such that $\sigma$ may be greater than 1, generalising the results of \cite{ArendtChill}, \cite{ArendtDierSystems} and \cite{LaasriMugnolo2020MMAS}. These results may provide useful applications, particularly in population models and advection-diffusion systems, as we try to exemplify in model problems.

\section{A Non-autonomous Linear Problem}\label{sect:MRQuasiLinear}

We first consider the linear problem for a system, for $0<s\leq1$ up to and including the classical case of $s=1$, as in \cite{ArendtDierSystems} but for a different maximal regularity space, by extending the approach of \cite{ArendtChill} to systems by introducing a suitable matrix $\Upsilon$.

We first observe that, by the definition of $\mathbb{A}$ as a symmetric time-independent operator, for $\bm{u}\in MR$, $\mathbb{A}\bm{u}(t)\in \mathbf{L}^2(\Omega)$ for a.e. $t\in ]0,T[$, so we have the following well-known result (see, for instance, \cite{DautrayLionsVol5}, page 480), which we include here for completeness.

\begin{lemma}\label{ProductRuleTime} Let $\bm{u}\in MR$. Then $\int_\Omega\mathbb{A}\bm{u}(\cdot)\cdot\bm{u}(\cdot)\in W^{1,1}(0,T)$ and \[\frac{d}{dt}\int_\Omega\mathbb{A}\bm{u}(t)\cdot\bm{u}(t)=2\int_\Omega\mathbb{A}\bm{u}(t)\cdot \bm{u}'(t)\quad\text{ for a.e. }t\in]0,T[.\]
Furthermore, the continuous embedding holds \[MR\hookrightarrow C([0,T];\mathbf{H}^s_0(\Omega)).\]
\end{lemma}

\begin{proof} We first take $\bm{u}\in C^1([0,T];\mathbf{L}^2_{\mathbb{A}})$. Then, since $\mathbb{A}$ is symmetric and time-independent, we have
\begin{equation*}\int_\Omega\mathbb{A}\bm{u}(t)\cdot \bm{u}'(t)=\frac12\left(\int_\Omega\mathbb{A}\bm{u}(t)\cdot \bm{u}'(t)+\int_\Omega \bm{u}(t)\cdot \mathbb{A}\bm{u}'(t)\right)=\frac12\frac{d}{dt}\int_\Omega\mathbb{A}\bm{u}(t)\cdot\bm{u}(t).\end{equation*} Then, the result holds for arbitrary $\bm{u}\in MR$ by an approximation by density.

The second part of the lemma may be proved as in Proposition 3.6 of \cite{Dier2015JMAANonAutonomousMR}. Indeed, since $\int_\Omega\mathbb{A}\bm{u}(\cdot)\cdot\bm{u}(\cdot)\in W^{1,1}(0,T)\subset C([0,T])$, together with the continuous embedding $MR\hookrightarrow C([0,T];\mathbf{L}^2(\Omega))$ and the coercivity \eqref{LEllipBdd} yields $MR\subset L^\infty(0,T;\mathbf{H}^s_0(\Omega))$. Now, it is well-known that $C([0,T];\mathbf{L}^2(\Omega))\cap L^\infty(0,T;\mathbf{H}^s_0(\Omega))\subset C([0,T];\mathbf{H}^s_0(\Omega)\text{-weak})$ (see, for instance, Lemma 3.3 of \cite{Dier2015JMAANonAutonomousMR}). Then, as $\tau\to t$ for fixed $t$, \begin{align*}a_*\norm{\bm{u}(t)-\bm{u}(\tau)}_{\mathbf{H}^s_0(\Omega)}^2&\leq\int_\Omega\mathbb{A}(\bm{u}(t)-\bm{u}(\tau))\cdot(\bm{u}(t)-\bm{u}(\tau))+\mu\norm{\bm{u}(t)-\bm{u}(\tau)}_{\mathbf{L}^2(\Omega)}^2\\&=2\int_\Omega\mathbb{A}(\bm{u}(t)-\bm{u}(\tau))\cdot\bm{u}(t)+\int_\Omega\left[\mathbb{A}\bm{u}(\tau)\cdot\bm{u}(\tau)-\mathbb{A}\bm{u}(t)\cdot\bm{u}(t)\right]\\&\quad+\mu\norm{\bm{u}(t)-\bm{u}(\tau)}_{\mathbf{L}^2(\Omega)}^2.\end{align*} The three terms tend to 0: the first one by the weak continuity of $\bm{u}(\cdot)$ in $\mathbf{H}^s_0(\Omega)$, the second one by the continuity of the map $\tau\mapsto \int_\Omega \mathbb{A}\bm{u}(\tau)\cdot\bm{u}(\tau)$ 
and the third one again by the embedding $MR\hookrightarrow C([0,T];\mathbf{L}^2(\Omega))$.
\end{proof}

Recall, for instance from \cite{DautrayLionsVol5}, the following well-known maximal regularity result: for all $\bm{f}\in L^2(0,T;\mathbf{L}^2(\Omega))$, $\bm{u}_0\in \mathbf{H}^s_0(\Omega)$, there exists a unique solution to the autonomous problem \begin{align}&\bm{u}\in H^1(0,T;\mathbf{L}^2(\Omega))\cap L^2(0,T;\mathbf{L}^2_{\mathbb{A}}),\nonumber\\& \bm{u}'(t)+\mathbb{A}\bm{u}(t)=\bm{f}(t)\quad\text{ for a.e. }t\in]0,T[,\label{LinAutonomousCauchyProb}\\&\bm{u}(0)=\bm{u}_0.\nonumber\end{align} 

We consider now a linear non-autonomous problem, obtained by a multiplicative perturbation.
\begin{theorem}\label{LinNonautonomousProbThm}
Let $\Upsilon=\Upsilon(t,x):]0,T[\times\Omega\to\mathbb{R}^{m\times m}$ is a measurable, coercive, invertible matrix satisfying \eqref{GammaNonDivCond}. Then, for every $\bm{f}\in L^2(0,T;\mathbf{L}^2(\Omega))$, $\bm{u}_0\in \mathbf{H}^s_0(\Omega)$, there exists a unique solution of the problem \begin{align}&\bm{u}\in H^1(0,T;\mathbf{L}^2(\Omega))\cap L^2(0,T;\mathbf{L}^2_{\mathbb{A}})\cap C([0,T];\mathbf{H}^s_0(\Omega)),\nonumber\\& \bm{u}'(t)+\Upsilon(t,\cdot)\mathbb{A}\bm{u}(t)=\bm{f}(t)\quad\text{ for a.e. }t\in]0,T[,\label{LinNonautonomousCauchyProb}\\&\bm{u}(0)=\bm{u}_0.\nonumber\end{align} Moreover, there exists a constant $c=c(\underline{\gamma},\bar{\gamma},a_*,a^*,\mu,T)>0$ independent of $\bm{f}$ and $\bm{u}_0$ such that \begin{equation}\label{EstLinNonautCauchy}\norm{\bm{u}}_{MR}\leq c\left(\norm{\bm{f}}_{L^2(0,T;\mathbf{L}^2(\Omega))}+\norm{\bm{u}_0}_{\mathbf{H}^s_0(\Omega)}\right)\end{equation} for each solution $\bm{u}$ of \eqref{LinNonautonomousCauchyProb}
\end{theorem}

\begin{proof}
We use the method of continuity (c.f. Section 5.2 of \cite{GilbargTrudingerBook}) as in Theorem 3.2 of \cite{ArendtChill}. For every $\lambda\in[0,1]$, consider the matrix $\Upsilon_\lambda:=(1-\lambda)\mathbb{I}+\lambda \Upsilon$ for the identity matrix $\mathbb{I}$ of dimension $m\times m$ and the bounded operator \[B_\lambda:MR\to L^2(0,T;\mathbf{L}^2(\Omega))\times \mathbf{H}^s_0(\Omega)\] given by \[B_\lambda \bm{u}=\int_\Omega (\bm{u}'+\Upsilon_\lambda \mathbb{A}\bm{u})\cdot\bm{u}_0.\] Then $B:[0,1]\to\mathcal{L}(MR, L^2(0,T;\mathbf{L}^2(\Omega))\times \mathbf{H}^s_0(\Omega))$ (where $\mathcal{L}$ denotes the space of linear bounded operators) is continuous and $B_0$ is invertible by the maximal regularity result for the linear autonomous problem \eqref{LinAutonomousCauchyProb}. Therefore, by Theorem 5.2 of \cite{GilbargTrudingerBook}, it suffices to prove the a priori estimate \begin{equation}\label{EstLinNonautCauchy2}\norm{\bm{u}}_{MR}\leq\norm{B_\lambda \bm{u}}= c_1\left(\norm{ \bm{u}'+\Upsilon_\lambda \mathbb{A}\bm{u}}_{L^2(0,T;\mathbf{L}^2(\mathbb{R}^n))}+\norm{\bm{u}_0}_{\mathbf{H}^s_0(\Omega)}\right)\,\forall\lambda\in [0,1],\forall \bm{u}\in MR,\end{equation} for some constant $c_1=c_1(\underline{\gamma},\bar{\gamma},a_*,a^*,T)>0$, which gives \eqref{EstLinNonautCauchy} for $\lambda=1$.

Let $\lambda\in[0,1]$. Let $\bm{u}\in MR$ be such that \[ \bm{u}'+\Upsilon_\lambda\mathbb{A}\bm{u}=\bm{f}\quad\text{ and }\quad \bm{u}(0)=\bm{u}_0.\] Then, multiplying the equation by $[\Upsilon_\lambda^*]^{-1}\bm{u}'(t)$, where $[\Upsilon_\lambda^*]^{-1}$ is the inverse of the adjoint of $\Upsilon_\lambda$, we have, for almost every $t\in[0,T]$, \[\int_\Omega [\Upsilon_\lambda^*]^{-1} \bm{u}'(t)\cdot \bm{u}'(t)+\int_\Omega \Upsilon_\lambda\mathbb{A}\bm{u} \cdot[\Upsilon_\lambda^*]^{-1}\bm{u}'(t)=\int_\Omega \bm{f}\cdot [\Upsilon_\lambda^*]^{-1} \bm{u}'(t),\] which by Lemma \ref{ProductRuleTime} and the Cauchy-Schwarz inequality, gives \[\int_\Omega [\Upsilon_\lambda^*]^{-1} \bm{u}'(t)\cdot \bm{u}'(t)+\frac12\frac{d}{dt}\int_\Omega\mathbb{A}\bm{u} \cdot\bm{u}'(t)\leq\frac{\bar{\gamma}}{2}\int_\Omega\left[[\Upsilon_\lambda^*]^{-1}\bm{f}(t)\right]^2+\frac{1}{2\bar{\gamma}}\int_\Omega \left[\bm{u}'(t)\right]^2.\]
Integrating over time on $]0,t[$ for every finite $t\in]0,T[$ and using the estimate \eqref{GammaNonDivCond}, it follows by \eqref{LEllipBdd} that \begin{equation*}\frac{1}{2\bar{\gamma}}\int_0^t\norm{ \bm{u}'(\tau)}_{\mathbf{L}^2(\Omega)}^2\,d\tau+\frac{a_*}{2}\norm{\bm{u}(t)}_{\mathbf{H}^s_0(\Omega)}^2\leq \frac{a^*}{2}\norm{\bm{u}_0}_{\mathbf{H}^s_0(\Omega)}^2+\frac{\mu}{2}\norm{\bm{u}(t)}_{\mathbf{L}^2(\Omega)}^2+\frac{\bar{\gamma}}{2\underline{\gamma}^2}\int_0^t\norm{\bm{f}(\tau)}_{\mathbf{L}^2(\Omega)}^2\,d\tau,\end{equation*}
Observe that by the Cauchy-Schwarz inequality, \begin{align}\norm{\bm{u}(t)}_{\mathbf{L}^2(\Omega)}^2&=\norm{\bm{u}(0)}_{\mathbf{L}^2(\Omega)}^2+\int_0^t\frac{d}{d\tau}\norm{\bm{u}(\tau)}_{\mathbf{L}^2(\Omega)}^2\,d\tau\nonumber\\&=\norm{\bm{u}_0}_{\mathbf{L}^2(\Omega)}^2+2\int_0^t\int_\Omega\bm{u}(\tau)\cdot \bm{u}'(\tau)\,d\tau\nonumber\\&\leq\norm{\bm{u}_0}_{\mathbf{L}^2(\Omega)}^2+2\mu\bar{\gamma}\int_0^t\norm{\bm{u}(\tau)}_{\mathbf{L}^2(\Omega)}^2\,d\tau+\frac{1}{2\mu\bar{\gamma}}\int_0^t\norm{ \bm{u}'(\tau)}_{\mathbf{L}^2(\Omega)}^2\,d\tau\label{L2HsnormIneq},\end{align} 
and so we have
\begin{multline}\frac{1}{2\bar{\gamma}}\int_0^t\norm{ \bm{u}'(\tau)}_{\mathbf{L}^2(\Omega)}^2\,d\tau+a_*\norm{\bm{u}(t)}_{\mathbf{H}^s_0(\Omega)}^2\\\leq a^*\norm{\bm{u}_0}_{\mathbf{H}^s_0(\Omega)}^2+\mu\norm{\bm{u}_0}_{\mathbf{L}^2(\Omega)}^2+\frac{\bar{\gamma}}{\underline{\gamma}^2}\norm{\bm{f}}_{L^2(0,T;\mathbf{L}^2(\Omega))}^2+2\mu^2\bar{\gamma}\int_0^t\norm{\bm{u}(\tau)}_{\mathbf{L}^2(\Omega)}^2\,d\tau\\\leq a^*\norm{\bm{u}_0}_{\mathbf{H}^s_0(\Omega)}^2+\mu\norm{\bm{u}_0}_{\mathbf{L}^2(\Omega)}^2+\frac{\bar{\gamma}}{\underline{\gamma}^2}\norm{\bm{f}}_{L^2(0,T;\mathbf{L}^2(\Omega))}^2+2\mu^2\bar{\gamma}c_S\int_0^t\norm{\bm{u}(\tau)}_{\mathbf{H}^s_0(\Omega)}^2\,d\tau\label{EstLinNonautCauchy3}\end{multline} where $c_S$ is the constant for the Sobolev embedding $\mathbf{H}^s_0(\Omega)\hookrightarrow\mathbf{L}^2(\Omega)$.
Applying the integral form of Gronwall's lemma to the second term on the left-hand-side, there exists a constant $c_2=c_2(\underline{\gamma},\bar{\gamma},a_*,a^*,\mu,T)>0$ such that  \[\sup_{t\in[0,T[}\norm{\bm{u}(t)}_{\mathbf{H}^s_0(\Omega)}^2\leq c_2\left(\norm{\bm{u}_0}_{\mathbf{H}^s_0(\Omega)}^2+\norm{\bm{f}}_{L^2(0,T;\mathbf{L}^2(\Omega))}^2\right).\] Inserting this into \eqref{EstLinNonautCauchy3}, we obtain that \[\int_0^T\norm{ \bm{u}'(\tau)}_{\mathbf{L}^2(\Omega)}^2\,d\tau\leq c_3\left(\norm{\bm{u}_0}_{\mathbf{H}^s_0(\Omega)}^2+\norm{\bm{f}}_{L^2(0,T;\mathbf{L}^2(\Omega))}^2\right)\] for some constant $c_3=c_3(\underline{\gamma},\bar{\gamma},a_*,a^*,\mu,T)>0$. 

Finally, since \[\int_0^T\norm{\mathbb{A}\bm{u}(\tau)}_{\mathbf{L}^2(\mathbb{R}^n)}^2\,d\tau\leq \frac{1}{\underline{\gamma}^2}\int_0^T\norm{\Upsilon_\lambda\mathbb{A}\bm{u}(\tau)}_{\mathbf{L}^2(\mathbb{R}^n)}^2\,d\tau= \frac{1}{\underline{\gamma}^2}\int_0^T\norm{ \bm{u}'(\tau)-\bm{f}(\tau)}_{\mathbf{L}^2(\Omega)}^2\,d\tau ,\] the $MR$ norm of $\bm{u}$ can be estimated giving \eqref{EstLinNonautCauchy} and by Lemma \ref{ProductRuleTime}, the proof is complete.

\end{proof}

\begin{remark}
If $\mu=0$ in \eqref{LEllipBdd}, this theorem is a special case of Theorem 1.1 of \cite{Bardos1971JFAParabolicReg}.
\end{remark}

Next, we identify local and nonlocal vectorial operators to which we can apply Theorem \ref{LinNonautonomousProbThm}.
\\

\emph{Example 1: Local operators}

As a first example of $\mathbb{A}$, we consider the local operator $\mathbb{L}$, such that the $i$-th component of $\mathbb{L}\bm{u}$ is given by
\begin{equation}\tag{\ref{ClassicalDeriv}}(\mathbb{L}\bm{u})^i=-\sum_{\alpha,\beta,j}\partial_\alpha(A^{\alpha\beta}_{ij}\partial_\beta u^j)+b_{ij}u^j\end{equation} for $b_{ij}\in L^\infty(\Omega)$ and $A^{\alpha\beta}_{ij}\in L^\infty(\Omega)$. 
Then, the linear non-autonomous problem for  $\bm{f}\in L^2(0,T;\mathbf{L}^2(\Omega))$ and $\bm{u}_0\in\mathbf{H}^1_0(\Omega)$ given by
\[\bm{u}'(t)+\Upsilon(t,\cdot)\mathbb{L}\bm{u}(t)=\bm{f}(t),\quad\text{ for a.e. }t\in]0,T[\] has a solution $\bm{u}\in H^1(0,T;\mathbf{L}^2(\Omega))\cap L^2(0,T;\mathbf{L}^2_{\mathbb{L}})$. This extends the results of \cite{ArendtChill}.

Furthermore, we can explicitly write out the space $\mathbf{L}^2_{\mathbb{L}}$ for the following special case of $\mathbb{A}=\mathbb{L}$ using the following proposition:

\begin{proposition}[Theorem 4.9 of \cite{GiaquintaMartinazziEllipticSystemsBook2ndEd}]\label{ClassicalEllipRegThmVectorial}
Let $\Omega$ be an open domain in $\mathbb{R}^n$. Suppose in addition that $A^{\alpha\beta}_{ij}\in C^{0,1}_{loc}(\Omega)$ is continuous up to the boundary of $\bar{\Omega}$ 
and satisfies the Legendre-Hadamard condition \begin{equation}\label{LegendreHadamard}\sum _{\alpha,\beta,i,j}A^{\alpha\beta}_{ij}\nu_\alpha\nu_\beta\xi^i\xi^j\geq a_*|\xi|^2|\nu|^2\quad\forall\xi\in\mathbb{R}^m,\nu\in\mathbb{R}^n.\end{equation} Then, for all weak solutions $\bm{u}$ of the equation with $\bm{f}\in\mathbf{L}^2_{loc}(\Omega)$, \[\mathbb{L}\bm{u}=\bm{f}\quad\text{ in }\Omega,\] $\bm{u}\in \mathbf{H}^2_{loc}(\Omega)$.

If, in addition, by Theorem 6 of \cite{DongKim2011ARMALpParabolicEllipticSystemsBMO}, $\Omega$ is bounded with $C^{1,1}$ boundary and $A^{\alpha\beta}_{ij}\in C^{0,1}(\bar{\Omega})$, we can extend the result globally up to the boundary of $\Omega$ for the unique solution of the homogeneous Dirichlet problem for $\bm{f}\in \mathbf{L}^2(\Omega)$, so that the unique solution $\bm{u}\in \mathbf{H}^1_0(\Omega)$ lies in $\mathbf{H}^2(\Omega)$.
\end{proposition}

It is well-known (see for instance, Section 5 of \cite{Fichera1973}) to show using Fourier transform that the Legendre-Hadamard condition \eqref{LegendreHadamard} for tensors $A$ continuous up to the boundary of $\bar{\Omega}$ implies coercivity \eqref{LEllipBdd}, which we recall, is given by G\r{a}rding's inequality \begin{equation}\tag{\ref{LEllipBdd}}\langle \mathbb{L}\bm{u},\bm{u}\rangle+\mu\norm{\bm{u}}_{\mathbf{L}^2(\Omega)}^2\geq a_*\norm{\bm{u}}_{\mathbf{H}^1_0(\Omega)}^2\quad\forall \bm{u}\in \mathbf{H}^1_0(\Omega).\end{equation} (Recall also that this is not true if $\bm{u}$ does not have support in $\bar{\Omega}$.)
Therefore, as a corollary, we have
\begin{corollary}
Suppose $\mathbb{L}$ is of the form \eqref{ClassicalDeriv} such that $A^{\alpha\beta}_{ij}$ is locally Lipschitz and continuous up to the boundary of $\bar{\Omega}$ satisfying the Legendre-Hadamard condition \eqref{LegendreHadamard}, then the linear non-autonomous Cauchy problem for $\bm{f}\in L^2(0,T;\mathbf{L}^2(\Omega))$ and $\bm{u}_0\in\mathbf{H}^1_0(\Omega)$ given by
\[\bm{u}'(t)+\Upsilon(t,\cdot)\mathbb{L}\bm{u}(t)=\bm{f}(t),\quad\text{ for a.e. }t\in]0,T[\] has a solution $\bm{u}\in H^1(0,T;\mathbf{L}^2(\Omega))\cap C([0,T];\mathbf{H}^1_0(\Omega))\cap L^2(0,T;\mathbf{H}^2_{loc}(\Omega))$.

If, in addition, $\Omega$ is bounded with $C^{1,1}$ boundary and $A^{\alpha\beta}_{ij}\in C^{0,1}(\bar{\Omega})$, then $\bm{u}\in H^1(0,T;\mathbf{L}^2(\Omega))\cap C([0,T];\mathbf{H}^1_0(\Omega))\cap L^2(0,T;\mathbf{H}^2(\Omega))$.
\end{corollary}

\vspace{1em}

\emph{Example 2: Anisotropic fractional operators}

Vectorial fractional operators $\tilde{\mathcal{L}}^s_A:\mathbf{H}^s_0(\Omega)\to\mathbf{H}^{-s}(\Omega)$ can also be considered, given for  $0<s\leq1$ by
\begin{equation}\tag{\ref{FracOp}}\langle\tilde{\mathcal{L}}^s_A\bm{u},\bm{v}\rangle=\sum_{\alpha,\beta,i,j}\int_{\mathbb{R}^n} A_{ij}^{\alpha\beta}D^s_\beta u^j\cdot D^s_\alpha v^i\end{equation} for a bounded, coercive tensor $A^{\alpha\beta}_{ij}$ symmetric in $\alpha,\beta$, i.e. 
\[\sum_{\alpha,i}a_*|\xi^\alpha_i|^2\leq \sum_{\alpha,\beta,i,j}A^{\alpha\beta}_{ij}\xi^\alpha_i\cdot\xi^\beta_j\leq a^*\sum_{\alpha,i}|\xi^\alpha_i|^2\quad\text{ for all }\xi\in\mathbb{R}^{m\times n}.\] Here, $D^s_\alpha$ coincides with the classical derivative $\partial_\alpha$ in the classical case of $s=1$, and $\tilde{\mathcal{L}}^s_A$ in \eqref{FracOp} reduces to $\mathbb{L}$ in \eqref{ClassicalDeriv} when $s=1$.

Then, the linear non-autonomous problem for  $\bm{f}\in L^2(0,T;\mathbf{L}^2(\Omega))$ and $\bm{u}_0\in\mathbf{H}^s_0(\Omega)$ given by
\[\bm{u}'(t)+\Upsilon(t,\cdot)\tilde{\mathcal{L}}^s_A\bm{u}(t)=\bm{f}(t),\quad\text{ for a.e. }t\in]0,T[\] has a solution $\bm{u}\in H^1(0,T;\mathbf{L}^2(\Omega))\cap L^2(0,T;\mathbf{L}^2_{\tilde{\mathcal{L}}^s_A})$.

In the particular case when $A^{\alpha\beta}$ is given by a diagonal constant matrix, $\tilde{\mathcal{L}}^s_A$ corresponds to a system of equations defined with the fractional Laplacian. 

Recall that the fractional Laplacian is defined, for $u\in H^s_0(\Omega)$, by \begin{equation}\label{FracLap}(-\Delta)^su(x):=c_{n,2s}\,P.V.\int_{\mathbb{R}^n}\frac{u(x)-u(y)}{|x-y|^{n+2s}}\,dy.\end{equation} Then,  by Theorem 7.1 of \cite{GrubbHormander}, or Theorem 4.1 and Remark 7 of \cite{borthagaray2021besov}, we have

\begin{proposition}\label{BNochettoFracLapReg}
Suppose $\Omega\subset\mathbb{R}^n$ is a bounded Lipschitz domain. Let $f\in L^2(\Omega)$ and $s\in]\frac12,1[$. Then the solution to the homogeneous Dirichlet problem \[(-\Delta)^su=f\text{ in }\Omega,\quad u=0 \text{ in }\Omega^c\] lies in the Besov space \[u\in \dot{B}^{s+1/2}_{2,\infty}(\Omega)\subset H^{\min\{2s,s+1/2\}-\epsilon}(\Omega)\] for any positive $\epsilon<\min\{2s,s+1/2\}$.
\end{proposition}

Considering the vectorial fractional Laplacian $(-\Delta)^s_m$ defined by 
\[(-\Delta)^s_m=\begin{bmatrix}
c_1(-\Delta)^s &  & 0\\
 & \ddots & \\
0 &  & c_m(-\Delta)^s
\end{bmatrix}\]
for constants $c_1,\cdots,c_m>0$. Then, applying Proposition \ref{BNochettoFracLapReg} component-wise, we have

\begin{corollary}
Suppose $\Omega\subset\mathbb{R}^n$ is a bounded Lipschitz domain. The linear non-autonomous Cauchy problem for $\bm{f}\in L^2(0,T;\mathbf{L}^2(\Omega))$ and $\bm{u}_0\in\mathbf{H}^s_0(\Omega)$ given by 
\[\bm{u}'(t)+\Upsilon(t,\cdot)(-\Delta)^s_m\bm{u}(t)=\bm{f}(t),\quad\text{ for a.e. }t\in]0,T[\] has a solution $\bm{u}\in H^1(0,T;\mathbf{L}^2(\Omega))\cap C([0,T];\mathbf{H}^s_0(\Omega))\cap L^2(0,T;\mathbf{H}^{\min\{2s,s+1/2\}-\epsilon}(\Omega))$
for any positive $\epsilon<\min\{2s,s+1/2\}$.

\end{corollary}
\vspace{1em}

\emph{Example 3: Anisotropic nonlocal operators}

Next, we consider the anisotropic nonlocal operator $\mathbb{L}^s_A:\mathbf{H}^s_0(\Omega)\to\mathbf{H}^{-s}(\Omega)$ \begin{equation}\tag{\ref{NonlocalOp}}\mathbb{L}^s_A\bm{u}=P.V.\int_{\mathbb{R}^n}A(x,y)\frac{\bm{u}(x)-\bm{u}(y)}{|x-y|^{n+2s}}\,dy\end{equation} defined for a symmetric, bounded, coercive matrix kernel $A(x,y)$
for $0<s<1$. Then, once again, the linear non-autonomous problem for  $\bm{f}\in L^2(0,T;\mathbf{L}^2(\Omega))$ and $\bm{u}_0\in\mathbf{H}^s_0(\Omega)$ given by
\[\bm{u}'(t)+\Upsilon(t,\cdot)\mathbb{L}^s_A\bm{u}(t)=\bm{f}(t),\quad\text{ for a.e. }t\in]0,T[\] has a solution $\bm{u}\in H^1(0,T;\mathbf{L}^2(\Omega))\cap L^2(0,T;\mathbf{L}^2_{\mathbb{L}^s_A})$.

Suppose in addition, $m=n$ and the kernel $A(x,y)$ is a measurable, translation-invariant matrix of the form \begin{equation}\label{kAcondSystem1}A(x,y)=\frac{\hat{a}(x-y)}{|x-y|^{n+2s}}\chi_{\mathcal{C}\cap B_r(0)}(x-y)\left(\frac{x-y}{|x-y|}\otimes\frac{x-y}{|x-y|}\right),\end{equation} where $\hat{a}$ is an even, coercive and bounded function such that $0<a_*\leq \hat{a}\leq a^*<\infty$ for some constants $a_*,a^*>0$, $0<r\leq\infty$, and $\mathcal{C}$ is a double cone with apex, i.e. \begin{multline*}\mathcal{C}=\bigg\{h\in\mathbb{R}^n\backslash\{0\}:\frac{h}{|h|}\in \mathcal{O}\cup(-\mathcal{O})\text{ for any open subset }\mathcal{O}\text{ of the unit sphere }S^{d-1}\\\text{ with positive Hausdorff measure}\bigg\}.\end{multline*}

Defining the space $\mathbf{H}^{s}_{loc}(\Omega)$ by $\{\bm{u}\in \mathbf{L}^2(\Omega):\eta\bm{u}\in \mathbf{H}^{s}(\Omega)\,\forall\eta\in C_c^\infty(\Omega)\}$, by Theorem 3.1 of \cite{KassmannMengeshaScott2019CPAASolvability}, we have the following local regularity result for $\mathbb{L}^s_A$ for this special case:

\begin{proposition}\label{KassmannMengeshaScottSystemsEllipReg}
Let $\Omega\subset\mathbb{R}^n$ be an open set, and $\bm{f}\in \mathbf{L}^2(\Omega)$ extended by 0 outside. Then the weak solution to \[\mathbb{L}^s_A\bm{u}=\bm{f}\text{ in }\Omega,\quad \bm{u}=\bm{0}\text{ in }\Omega^c\] lies in $\mathbf{H}^{2s}_{loc}(\Omega)$. Moreover, for any $\eta\in C_c^\infty(\Omega)$, there exists a constant $C$ such that \[\norm{\eta \bm{u}}_{\mathbf{H}^{2s}(\mathbb{R}^n)}\leq C\norm{\bm{f}}_{\mathbf{L}^2(\Omega)}.\]
\end{proposition}

\begin{remark}
Observe that $\mathbb{L}_A^s$ can be viewed as the nonlocal version of Example 1, as explained in pages 1304--1305 of \cite{KassmannMengeshaScott2019CPAASolvability}. See also Lemma 3.1 of \cite{FelsingerKassmannVoigt}.
\end{remark}

As a corollary, we once again have
\begin{corollary}
The linear non-autonomous Cauchy problem for $\bm{f}\in L^2(0,T;\mathbf{L}^2(\Omega))$ and $\bm{u}_0\in\mathbf{H}^s_0(\Omega)$ given by
\[\bm{u}'(t)+\Upsilon(t,\cdot)\mathbb{L}^s_A\bm{u}(t)=\bm{f}(t),\quad\text{ for a.e. }t\in]0,T[\] has a solution $\bm{u}\in H^1(0,T;\mathbf{L}^2(\Omega))\cap C([0,T];\mathbf{H}^s_0(\Omega))\cap L^2(0,T;\mathbf{H}^{2s}_{loc}(\Omega))$.

\end{corollary}

\section{The Nonlinear Problem $\sigma\leq s\leq1$}\label{sect:MRQuasiSigmaSmall}

We next consider the quasilinear vectorial problem, when $0<\sigma< s\leq1$, extending the nonlocal vectorial problem with no source function considered in \cite{LaasriMugnolo2020MMAS}, as well as the vectorial semilinear case in \cite{Amann1985GlobalExistenceSemilinearParabolicSystems}, \cite{Morgan1989GlobalExistenceSemilinearParabolicSystems} and \cite{ArendtDierSystems}. This also generalises \cite{ArendtChill} to systems of the form \eqref{MainProblemGlobalQuasilinear} defined in a bounded or unbounded open set $\Omega\subset\mathbb{R}^n$. We will apply the Schaefer fixed point theorem, which is a generalisation of the Leray-Schauder fixed point theorem to locally convex spaces, to the approximating bounded subsets $\Omega_k\subset\mathring{\Omega}$, as in \cite{ArendtChill}, so that the regularity of the boundary of $\Omega$ can be ignored.

Assume \begin{equation}\label{DomainAssump}D(\mathbb{A})=\mathbf{L}^2_{\mathbb{A}}\subset \mathbf{H}^{s+\theta}_{loc}(\Omega)\quad\text{ for some }\theta\geq0.\end{equation} Note that this assumption is weaker than the one given in Equation (4.2) of \cite{ArendtChill}, and allows us to cover the fractional derivatives as well.

Then, we have the following main result:
\begin{theorem}\label{NonlinNonautonomousProbThmSupercritical} 
Suppose $\Omega\subset\mathbb{R}^n$ is an open set.
Let $\mathbb{A}$ satisfy \eqref{DomainAssump} for $\theta\geq0$, and \[\Pi:]0,T[\times\Omega\times\mathbb{R}^m\times\mathbb{R}^q\to\mathbb{R}^{m\times m}\] be a measurable, invertible matrix $\Pi=\Pi(t,x,\bm{u},\bm{p})$ satisfying \begin{equation}\tag{\ref{GammaNonDivCond}}\underline{\gamma}|\xi|^2\leq\Pi\xi\cdot\xi\quad\text{ and }\quad\Pi\xi\cdot\xi^*\leq\bar{\gamma}|\xi||\xi^*|,\quad 0<\underline{\gamma}\leq\bar{\gamma},\quad\text{ for all }\xi,\xi^*\in\mathbb{R}^m,\end{equation} such that $\Pi$ is continuous in $\bm{u}$ and $\bm{p}$ for almost every $(t,x)$. Let  \[\bm{f}:]0,T[\times\Omega\times\mathbb{R}^m\times\mathbb{R}^q\to\mathbb{R}^m\] be a measurable vector function that is continuous in $\bm{u}$ and $\bm{p}$ for almost every $(t,x)$, satisfying \begin{equation}\label{fcondNonlin}|\bm{f}(t,x,\bm{u},\bm{p})|\leq F(t,x)+\Lambda_1|\bm{u}|+\Lambda_2|\bm{p}| \quad\text{ for some }F\in L^2(0,T;L^2(\Omega)), \Lambda_1,\Lambda_2\geq0.\end{equation} Then for every $\bm{u}_0$ such that $\bm{u}_0\in \mathbf{H}^s_0(\Omega)$, there exists  \begin{equation}\label{QuasilinNonautonomousSupercriticalSolnReg}\bm{u}\in H^1(0,T;\mathbf{L}^2(\Omega))\cap L^2(0,T;\mathbf{L}^2_{\mathbb{A}})\cap C([0,T];\mathbf{H}^s_0(\Omega)), \end{equation} solving the problem \begin{equation}\label{QuasilinNonautonomousProbSupercritical}\begin{split}& \bm{u}'(t)+\Pi(t,x,\bm{u},\Sigma \bm{u})\mathbb{A}\bm{u}(t)=\bm{f}(t,x,\bm{u},\Sigma \bm{u})\quad\text{ for a.e. }t\in]0,T[,\\&\bm{u}(0)=\bm{u}_0\end{split}\end{equation}  where $\Sigma$ represents fractional derivatives of order $\sigma\leq s$ with $0<\sigma<s+\theta$ for $0<s\leq1$. Moreover, there exists a constant $c'=c'(\underline{\gamma},\bar{\gamma},a_*,a^*,\Lambda_1,\Lambda_2,T)>0$ such that for every solution $\bm{u}$ of \eqref{QuasilinNonautonomousProbSupercritical}, \begin{equation}\label{EstQuasilinNonautCauchySupercritical}\norm{\bm{u}}_{MR}\leq c'\left(\norm{F}_{L^2(0,T;L^2(\Omega))}+\norm{\bm{u}_0}_{\mathbf{H}^s_0(\Omega)}\right).\end{equation}
\end{theorem}

\begin{remark}
In general, this solution is not unique.
\end{remark}

\begin{remark}
This extends the results for the classical derivatives with $s=\theta=1$ so that $s+\theta=2$ with $\sigma=1$ as considered in \cite{ArendtChill} for the scalar problem, as well as \cite{ArendtDierSystems} and \cite{LaasriMugnolo2020MMAS} for the semilinear vectorial problem and quasilinear vectorial problem respectively. In particular, we can consider fractional or nonlocal derivatives of any order $\sigma\leq s\leq1$. This generalises the classical gradient, and is conceptually similar to the ideas of Boussandel (see \cite{Boussandel2011JDE} and \cite{Boussandel2019CzechMath}), where he considers the classical gradient weighted by a measure.

For general operators $\mathbb{A}$ satisfying \eqref{LEllipBdd}, the theorem applies with $\theta=0$ and $\sigma<s\leq1$. For special operators satisfying \eqref{DomainAssump} with $\theta>0$, we can consider derivatives of order $\sigma$ up to and including $\sigma=s$ for $s\leq 1$, as in Section \ref{subsect:MRQuasiLinearAppPopulation}.
This includes the classical vectorial operator $\mathbb{L}$ with $s=\theta=1$ as given in Proposition \ref{ClassicalEllipRegThmVectorial}, as well as the nonlocal vectorial operator $\mathbb{L}^s_A$ with $\theta=s<1$ in Proposition \ref{KassmannMengeshaScottSystemsEllipReg}. Observe that the latter includes the case of the fractional Laplacian $(-\Delta)^s_m$ of Proposition \ref{BNochettoFracLapReg}, with $0<\theta<\frac12$ for $s>\frac12$ and $\theta<s$ for $s<\frac12$. Furthermore, when $\mathbb{A}=(-\Delta)^s_m$, $\Sigma \bm{u}$ can involve both the fractional and the nonlocal derivatives. This is because the fractional Laplacian $(-\Delta)^s$ can be represented by both the fractional derivatives $D^s$ and the nonlocal derivatives $\mathcal{D}^s$ i.e. $(-\Delta)^su=-D^s\cdot D^su=c_{n,s}^s\mathcal{D}_s\cdot\mathcal{D}^su$ when considered in $\mathbb{R}^n$. 
\end{remark}

We shall also use the Schaefer fixed point theorem, as it is reproduced in  Theorem 2.2 of \cite{arendt2019galerkin}, which we state here for reference. Note that if $E$ is a Banach space, this theorem reduces to the Leray-Schauder fixed point theorem (see, for instance, Theorem 11.3 of \cite{GilbargTrudingerBook}).
\begin{theorem}[Schaefer Fixed Point Theorem]\label{Schaefer}
Let $E$ be a complete locally convex vector space and let $S:E\to E$ be a continuous mapping. Assume that there exists a continuous seminorm $p:E\to\mathbb{R}^+$, a constant $R>0$, and a compact set $\mathcal{K}\subset E$ such that the Schaefer set \[\mathscr{S}=\{\bm{u}\in E:\bm{u}=\lambda S\bm{u}\text{ for some }\lambda\in[0,1]\}\] is included in \[\mathcal{C}:=\{\bm{u}\in E:p(\bm{u})<R\}\] such that \[S\mathcal{C}\subset\mathcal{K}.\] Then $S$ has a fixed point.
\end{theorem}

Recalling that $\Sigma \bm{u}\in \mathbb{R}^q$ for $0<q\leq m\times n$ represents fractional or nonlocal derivatives in the form $D^\sigma \bm{u}$ or $\mathcal{D}^\sigma \bm{u}$, we observe that for any Lipschitz bounded open set $\mathcal{O}$ such that $\bar{\mathcal{O}}\subset\Omega\subset\mathbb{R}^n$, for every $\bm{v}\in L^2(0,T;\mathbf{H}^\sigma(\mathcal{O}))$, the extension $\tilde{\bm{v}}\in L^2(0,T;\mathbf{H}^\sigma(\mathbb{R}^n))$ and \begin{equation}\label{SigmaNormBounds}\norm{\Sigma\tilde{\bm{v}}}_{L^2(0,T;\mathbf{L}^2(\mathcal{O}))}\leq C\norm{\bm{v}}_{L^2(0,T;\mathbf{H}^\sigma(\mathcal{O}))}\end{equation} for some constant $C$ depending on $\mathcal{O}$. 

\subsection{Proof of Theorem \ref{NonlinNonautonomousProbThmSupercritical}}

Let $(\Omega_k)_k$ be an increasing sequence of open bounded subsets of $\mathbb{R}^n$ with Lipschitz boundaries such that $\overline{\Omega_k}\subset\Omega$ and $\bigcup_{k\in\mathbb{N}}\Omega_k=\Omega$. 
Consider the locally convex space \begin{align*} E&:=L^2(0,T;\mathbf{H}^\sigma_{loc}(\Omega))\\&:=\{\bm{u}\in \mathbf{L}^2_{loc}(]0,T[\times\Omega):\bm{u}|_{]0,T[\times\Omega_k}\in L^2(0,T;\mathbf{H}^\sigma(\Omega_k))\text{ for every }k\in\mathbb{N}\},\end{align*} which is a Fr\'echet space for the sequence of seminorms given by $\norm{\cdot}_{L^2(0,T;\mathbf{H}^\sigma(\Omega_k)}$, $k\in\mathbb{N}$, as defined in \eqref{HsRestrictedNorm} for each $\Omega_k$.

Recall that for a Lipschitz open bounded set $\mathcal{O}\subset\mathbb{R}^n$ (c.f. Theorem 7.26 of \cite{GilbargTrudingerBook}), the Sobolev embedding $\mathbf{H}^{\sigma'}(\mathcal{O})\hookrightarrow \mathbf{H}^\sigma(\mathcal{O})$ is compact for $\sigma<\sigma'$ by the Rellich-Kondrachov theorem. Then, by Aubin-Lions lemma (Lemma II.7.7 of \cite{RoubicekBook}), we have the compact embedding \begin{equation}\label{CompactEmbeddingGeneral}H^1(0,T;\mathbf{L}^2(\mathcal{O}))\cap L^2(0,T;\mathbf{H}^{\sigma'}(\mathcal{O}))\hookrightarrow L^2(0,T;\mathbf{H}^\sigma(\mathcal{O})),\end{equation} and this embedding is continuous by the closed graph theorem. Since $\mathbf{L}^2_{\mathbb{A}}\subset \mathbf{H}^{s+\theta}_0(\Omega)$ by Assumption \eqref{DomainAssump}, applying \eqref{CompactEmbeddingGeneral} for $\sigma'=s+\theta$ for $\theta\geq0$ for the open bounded sets $\Omega_k$, it follows that \begin{equation}\label{MRECptEmbedSupercritical}MR=H^1(0,T;\mathbf{L}^2(\Omega))\cap L^2(0,T;\mathbf{L}^2_{\mathbb{A}})\hookrightarrow L^2(0,T;\mathbf{H}^\sigma_{loc}(\Omega))=E\end{equation} for $\sigma<s+\theta$ is also compact, for any set $\Omega\subseteq\mathbb{R}^n$.

Fix $T$ and $\bm{u}_0\in \mathbf{H}^s_0(\Omega)$. We first show the result for every $k$, i.e. for every $k$, the problem for $\bm{u}=\bm{u}_k$
\begin{align}&\bm{u}\in MR=H^1(0,T;\mathbf{L}^2(\Omega))\cap L^2(0,T;\mathbf{L}^2_{\mathbb{A}}),\nonumber\\& \bm{u}'(t)+\Pi(t,x,\bm{u},\Sigma \bm{u})\mathbb{A}\bm{u}(t)=\chi_{\Omega_k}(x)\bm{f}(t,x,\bm{u},\Sigma \bm{u})\quad\text{ for a.e. }t\in]0,T[,\label{QuasilinNonautonomousProbOmegak}\\&\bm{u}(0)=\bm{u}_0\nonumber\end{align} admits a solution such that for every solution $\bm{u}$, we have \begin{equation}\tag{\ref{EstQuasilinNonautCauchySupercritical}}\norm{\bm{u}}_{MR}\leq c'\left(\norm{F}_{L^2(0,T;L^2(\Omega))}+\norm{\bm{u}_0}_{\mathbf{H}^s_0(\Omega)}\right)\end{equation} for some constant $c'=c'(\underline{\gamma},\bar{\gamma},a_*,a^*,\mu,\Lambda_1,\Lambda_2,T)>0$ independent of $k$. Here $\chi_{\Omega_k}(x)$ is the scalar characteristic function which is 1 if $x\in\Omega_k$ and 0 otherwise.

For each fixed $k\in\mathbb{N}$ and for every $\bm{v}\in E$, we set \[\Pi_{\bm{v}}(t,x):=\Pi(t,x,\bm{v}(t,x),\Sigma \tilde{\bm{v}}(t,x)),\quad\text{ and }\]\[\bm{f}_{\bm{v},k}(t,x):=\chi_{\Omega_k}(x)\bm{f}(t,x,\bm{v}(t,x),\Sigma \tilde{\bm{v}}(t,x)).\] Then, $\Pi_{\bm{v}}$ inherits the same properties as $\Pi$, while $\bm{f}_{\bm{v},k}$ is measurable and satisfies \begin{align*}\norm{\bm{f}_{\bm{v},k}}_{L^2(0,T;\mathbf{L}^2(\Omega))}^2&\leq \int_0^T\int_{\Omega_k} F(t,x)+\Lambda_1^2|\bm{v}|^2+\Lambda_2^2|\Sigma \bm{v}|^2\\&\leq \norm{F}_{L^2(0,T;L^2(\Omega))}^2+\Lambda_1^2\norm{\bm{v}}_{L^2(0,T;\mathbf{L}^2(\Omega_k))}^2+C_k^2\Lambda_2^2\norm{\bm{v}}_{L^2(0,T;\mathbf{H}^\sigma(\Omega_k))}^2<\infty\end{align*} for the same constant $C_k=C(\Omega_k)$ as in \eqref{SigmaNormBounds}.
Then, by Theorem \ref{LinNonautonomousProbThm}, there exists a unique solution $\bm{u}=:\mathcal{T}_k\bm{v}\in MR$ of the problem \begin{align}& \bm{u}'(t)+\Pi_{\bm{v}}(t,\cdot)\mathbb{A}\bm{u}(t)=\bm{f}_{\bm{v},k}(t,\cdot)\quad\text{ for a.e. }t\in]0,T[\label{LimitingProbNonlin},\text{ and }\\&\bm{u}(0)=\bm{u}_0\nonumber\end{align} such that $\bm{u}\in MR$, satisfying the inequality \begin{align}\norm{\bm{u}}_{MR}&\leq c\left(\norm{\bm{f}_{\bm{v},k}}_{L^2(0,T;\mathbf{L}^2(\Omega))}+\norm{\bm{u}_0}_{\mathbf{H}^s_0(\Omega)}\right)\nonumber\\&\leq c_4\left(\norm{F}_{L^2(0,T;L^2(\Omega))}^2+\norm{\bm{v}}_{L^2(0,T;\mathbf{H}^\sigma(\Omega_k))}^2+\norm{\bm{u}_0}_{\mathbf{H}^s_0(\Omega)}\right)\label{EstNonlin1}\end{align} for some constant $c_4=c_4(c,k,\Lambda_1,\Lambda_2)$, where $c$ is the same constant from Theorem \ref{LinNonautonomousProbThm}. In this way, we have defined an operator $\mathcal{T}_k:E\to MR\subset E$.

Next, let $\bm{v}_i\to \bm{v}$ in $E$, and denote $\bm{u}_i=\mathcal{T}_k\bm{v}_i$ and $\bm{u}=\mathcal{T}_k\bm{v}$. We want to show that $\mathcal{T}_k$ is continuous, i.e. $\bm{u}_i\to \bm{u}$ in $E$. Since $(\bm{u}_i)_i$ is bounded in $MR$ by the estimate \eqref{EstNonlin1} which is uniform in $i$ for fixed $k$, and since $MR$ is a Hilbert space, we may assume, after passing to a subsequence, that there exists a $\bm{w}\in E$ such that \begin{equation}\label{ConvgNonlin1}\bm{u}_i\rightharpoonup \bm{w}\quad\text{ in }MR.\end{equation} Passing to a further subsequence, we may in addition assume that \begin{equation}\label{ConvgNonlin2} \bm{u}'_i\rightharpoonup \bm{w}'\quad\text{ in }L^2(0,T;\mathbf{L}^2(\Omega)),\text{ and }\end{equation}\begin{equation}\label{ConvgNonlin3}\mathbb{A}\bm{u}_i\rightharpoonup\mathbb{A}\bm{w}\quad\text{ in }L^2(0,T;\mathbf{L}^2(\Omega)).\end{equation} 
We show that $\bm{w}=\bm{u}$. Since $\bm{v}_i\to \bm{v}$ in $E$, passing to a further subsequence and using a diagonalisation argument, there exists a function $V_k\in L^2(]0,T[\times\Omega_k)$ such that \begin{equation}\label{UnifDomV}\begin{split}(\bm{v}_i,\Sigma \bm{v}_i)\to(\bm{v},\Sigma \bm{v})\quad&\text{ a.e. on }]0,T[\times\Omega,\quad\forall i\in\mathbb{N},\text{ and}\\|\bm{v}_i|+|\Sigma \bm{v}_i|\leq V_k \quad&\text{ a.e. on }]0,T[\times\Omega_k,\quad\forall i\in\mathbb{N}
\end{split}
\end{equation} by the continuity of $\Sigma$ which involves the $\partial$, $D^s$ and $\mathcal{D}^s$ operators. 

By the continuity of $\Pi$ and $\bm{f}$, we have \[\Pi_{\bm{v}_i}(t,x):=\Pi(t,x,\bm{v}_i,\Sigma \bm{v}_i)\to\Pi(t,x,\bm{v},\Sigma \bm{v})=:\Pi_{\bm{v}}(t,x),\quad\text{ and }\]\[\bm{f}_{\bm{v}_i,k}(t,x):=\chi_{\Omega_k}(x)\bm{f}(t,x,\bm{v}_i,\Sigma \bm{v}_i)\to \chi_{\Omega_k}(x)\bm{f}(t,x,\bm{v},\Sigma \bm{v})=:\bm{f}_{\bm{v},k}(t,x)\quad\text{ a.e. on }]0,T[\times\Omega.\]
Moreover, by the growth assumption on $\bm{f}$ in \eqref{fcondNonlin} and uniform domination of $\bm{v}_i$ by $V_k$ in \eqref{UnifDomV}, we have \begin{equation}\label{fdomination}|\bm{f}_{\bm{v}_i,k}|\leq F+(\Lambda_1+\Lambda_2) V_k \quad\text{ a.e. in }]0,T[\times\Omega_k,\quad\forall i\in\mathbb{N}.\end{equation} 
Recall that, for every $i\in\mathbb{N}$, $\bm{u}_i$ satisfies the problem \begin{equation}\label{NonlinProbSeq} \bm{u}'_i+\Pi_{\bm{v}_i}\mathbb{A}\bm{u}_i=\bm{f}_{\bm{v}_i}.\end{equation} By the Dominated Convergence Theorem and \eqref{fdomination}, \[\bm{f}_{\bm{v}_i,k}\to \bm{f}_{\bm{v},k}\quad\text{ strongly in }L^2(0,T;\mathbf{L}^2(\Omega)).\] 
Also, by the Dominated Convergence Theorem, since $\Pi_{\bm{v}_i}$ is uniformly bounded as in \eqref{GammaNonDivCond}, we have, for every $\varphi\in L^2(0,T;\mathbf{L}^2(\Omega))$,  \begin{equation*}\Pi_{\bm{v}_i}^*\varphi\to \Pi_{\bm{v}}^* \varphi\quad\text{ in }L^2(]0,T[\times\Omega).\end{equation*} 
By \eqref{ConvgNonlin3}, it follows that for every $\varphi\in L^2(0,T;\mathbf{L}^2(\Omega))$, \[\int_0^T\int_{\Omega} \Pi_{\bm{v}_i}\mathbb{A}\bm{u}_i\cdot\varphi=\int_0^T\int_{\Omega} \mathbb{A}\bm{u}_i\cdot\Pi_{\bm{v}_i}^*\varphi\to \int_0^T\int_{\Omega} \mathbb{A}\bm{w}\cdot\Pi_{\bm{v}}^*\varphi=\int_0^T\int_{\Omega} \Pi_{\bm{v}}\mathbb{A}\bm{w}\cdot\varphi,\] or equivalently \begin{equation}\label{ConvgNonlin4}\Pi_{\bm{v}_i}\mathbb{A}\bm{u}_i\rightharpoonup \Pi_{\bm{v}} \mathbb{A}\bm{w}\quad\text{ weakly in }L^2(0,T;\mathbf{L}^2(\Omega)).\end{equation} 
Therefore, taking $i\to\infty$ in \eqref{NonlinProbSeq} gives \[\bm{w}'(t)+\Pi_{\bm{v}}\mathbb{A}\bm{w}(t)=\bm{f}_{\bm{v},k}(t)\quad\text{ in }\Omega\text{ for a.e. }t\in]0,T[.\] Since $MR\hookrightarrow C([0,T];\mathbf{H}^s_0(\Omega))$ by Lemma \ref{ProductRuleTime}, the weak convergence of $\bm{w}\rightharpoonup\bm{u}$ in $MR$ gives \[\bm{w}(0)=\lim_{i\to\infty}\bm{u}_i(0)=\bm{u}_0.\] But $\bm{u}$ is also the solution of the problem \eqref{LimitingProbNonlin} which is unique by Theorem \ref{LinNonautonomousProbThm}, so $\bm{w}=\bm{u}$. 

Since $\bm{u}_i\rightharpoonup \bm{u}$ in $MR$, by the compact embedding $MR\hookrightarrow E$, we obtain, passing to a subsequence if necessary, \[\bm{u}_i\to \bm{u}\quad\text{ strongly in }E,\] so $\mathcal{T}_k$ is continuous.

In the next step, we show that there exists a non-negative constant depending on $\underline{\gamma},\bar{\gamma},a_*,a^*,\mu,\Lambda_1,\Lambda_2$ and $T$ independent of $k$ such that for every element $\bm{u}$ in the Schaefer set \[\mathscr{S}_k=\{\bm{v}\in E:\bm{v}=\alpha \mathcal{T}_k \bm{v} \text{ for some }\alpha\in[0,1]\},\] the estimate \eqref{EstQuasilinNonautCauchySupercritical} holds. 

Assume that $\bm{u}=\alpha\mathcal{T}_k(\bm{u})$ for some $\alpha\in[0,1]$, i.e. $\bm{u}$ satisfies \begin{equation}\label{ApproxEqOmegak}\begin{split}& \bm{u}'(t)+\Pi(t,\cdot,\bm{u},\Sigma \bm{u})\mathbb{A}\bm{u}(t)=\alpha \chi_{\Omega_k}(x)\bm{f}(t,\cdot,\bm{u},\Sigma \bm{u})\quad\text{ for a.e. }t\in]0,T[,\text{ and }\\&\bm{u}(0)=\bm{u}_0.\end{split}\end{equation} 

Multiplying the equation by $[\Pi_{\bm{u}}^*]^{-1}\bm{u}'(t)$ and integrating over $\Omega$, we obtain, by Lemma \ref{ProductRuleTime} and the Cauchy-Schwarz inequality, 
\begin{multline*}\int_\Omega [\Pi_{\bm{u}}^*]^{-1} \bm{u}'(t)\cdot \bm{u}'(t)+\frac{1}{2}\frac{d}{dt}\int_\Omega\mathbb{A}\bm{u}\cdot\bm{u}=\int_\Omega [\Pi_{\bm{u}}^*]^{-1} \bm{u}'(t)\cdot \bm{u}'(t)+\int_\Omega\mathbb{A}\bm{u}\cdot\bm{u}'(t)\\=\alpha\int_{\Omega_k} \bm{f}\cdot [\Pi_{\bm{u}}^*]^{-1} \bm{u}'(t)\leq\frac{\bar{\gamma}}{2}\int_\Omega|[\Pi_{\bm{u}}^*]^{-1}\bm{f}|^2+\frac{1}{2\bar{\gamma}}\int_\Omega |\bm{u}'(t)|^2\leq\frac{\bar{\gamma}}{2\underline{\gamma}^2}\int_\Omega|\bm{f}|^2+\frac{1}{2\bar{\gamma}}\int_\Omega |\bm{u}'(t)|^2,\end{multline*}  by the positivity of $\Pi$ and since $\alpha\in[0,1]$.
Making use of the coercivity and boundedness of $\Pi_{\bm{u}}$ in \eqref{GammaNonDivCond}, we integrate over time on $]0,t[$ for every finite $t\in]0,T[$ to obtain, by \eqref{LEllipBdd} and \eqref{L2HsnormIneq}, that  \begin{equation}\label{EstNonlin2}\begin{split}&\,\frac{1}{\bar{\gamma}}\int_0^t\norm{ \bm{u}'(\tau)}_{\mathbf{L}^2(\Omega)}^2\,d\tau+a_*\norm{\bm{u}(t)}_{\mathbf{H}^s_0(\Omega)}^2\\\leq&\, a^*\norm{\bm{u}_0}_{\mathbf{H}^s_0(\Omega)}^2+\mu\norm{\bm{u}(t)}_{\mathbf{L}^2(\Omega)}^2+\frac{1}{2\underline{\gamma}}\int_0^t\norm{\bm{f}(\tau)}_{\mathbf{L}^2(\Omega)}^2\,d\tau\\\leq&\, a^*\norm{\bm{u}_0}_{\mathbf{H}^s_0(\Omega)}^2+2\mu^2\bar{\gamma}\int_0^t\norm{\bm{u}(\tau)}_{\mathbf{L}^2(\Omega)}^2\,d\tau+\frac{1}{2\bar{\gamma}}\int_0^t\norm{ \bm{u}'(\tau)}_{\mathbf{L}^2(\Omega)}^2\,d\tau\\&\quad+\frac{\bar{\gamma}}{2\underline{\gamma}^2}\int_0^t\norm{F(\tau)}_{L^2(\Omega)}^2\,d\tau+\frac{\Lambda_1^2\bar{\gamma}}{2\underline{\gamma}^2}\int_0^t\norm{\bm{u}(\tau)}_{\mathbf{L}^2(\Omega)}^2\,d\tau+\frac{\Lambda_2^2\bar{\gamma}}{2\underline{\gamma}^2}\int_0^t[\bm{u}(\tau)]_{\mathbf{H}^\sigma(\mathbb{R}^n)}^2\,d\tau\\\leq&\, a^*\norm{\bm{u}_0}_{\mathbf{H}^s_0(\Omega)}^2+\frac{\bar{\gamma}}{2\underline{\gamma}^2}\norm{F}_{L^2(0,T;L^2(\Omega))}^2\\&\quad+\left(2\mu^2\bar{\gamma}+\bar{\gamma}\frac{\Lambda_1^2+c_S\Lambda_2^2}{2\underline{\gamma}^2}\right)\int_0^t\norm{\bm{u}(\tau)}_{\mathbf{H}^{s}_0(\Omega)}^2\,d\tau+\frac{1}{2\bar{\gamma}}\int_0^t\norm{ \bm{u}'(\tau)}_{\mathbf{L}^2(\Omega)}^2\,d\tau\end{split}\end{equation} by the Sobolev embedding $\mathbf{H}^\sigma_0(\Omega)\hookrightarrow\mathbf{H}^s_0(\Omega)$ with Sobolev constant $c_S$ for $\sigma\leq s$. Then, applying Gronwall's lemma, we can argue as in the proof of Theorem \ref{LinNonautonomousProbThm} to get the estimate \eqref{EstQuasilinNonautCauchySupercritical} for every $\bm{u}\in\mathscr{S}_k$.

This means that $\mathscr{S}_k$ is bounded in $MR$. By the definition of the $MR$ norm, this implies that there exists an $R>0$ such that \[\mathscr{S}_k\subset\mathcal{C}_k:=\{\bm{v}\in E:\norm{\bm{v}}_{L^2(0,T;\mathbf{H}^\sigma(\Omega_k))}<R\},\] because clearly  $\norm{\cdot}_{L^2(0,T;\mathbf{H}^\sigma(\Omega_k))}\leq\norm{\cdot}_{L^2(0,T;\mathbf{H}^\sigma(\Omega))}$.
It follows from the definition of $\mathcal{T}_k$ and \eqref{EstNonlin1} that $\mathcal{T}_k\mathcal{C}_k$ is contained in a bounded subset of $MR$. By compactness of the embedding \eqref{MRECptEmbedSupercritical}, $\mathcal{T}_k\mathcal{C}_k$ is contained in a compact subset of $E$. Therefore, by Schaefer's fixed point theorem (Theorem \ref{Schaefer}), the mapping $\mathcal{T}_k$ admits a fixed point $\bm{u}$ such that $\bm{u}\in MR$. By the definition of $\mathcal{T}_k$, this element $\bm{u}$ is a solution of the problem \eqref{QuasilinNonautonomousProbOmegak}, and since $\bm{u}\in\mathscr{S}_k$, $\bm{u}$ satisfies \eqref{EstQuasilinNonautCauchySupercritical}.

Finally, we extend the result to show that \eqref{QuasilinNonautonomousProbSupercritical} admits a solution. For every $k\in\mathbb{N}$, we choose a solution $\bm{u}_k$ of the problem \eqref{QuasilinNonautonomousProbOmegak}. Since every such solution is an element of $\mathscr{S}_k$ and satisfies the estimate \eqref{EstQuasilinNonautCauchySupercritical} which is independent of $k$, the sequence $(\bm{u}_k)_k$ is bounded in $MR$. Since $MR$ is a Hilbert space, we may assume (after passing to a subsequence) that there exists a limit $\bm{u}\in E$ such that $\bm{u}_k\rightharpoonup \bm{u}$ in $MR$. By the compactness of the embedding \eqref{MRECptEmbedSupercritical}, passing to a subsequence again if necessary, we obtain, through a diagonalisation argument, that
\begin{equation}\label{OmegakOmegaConvg}\begin{split}&\bm{u}'_k\rightharpoonup \bm{u}'\quad\text{ in }L^2(0,T;\mathbf{L}^2(\Omega)),\\&\mathbb{A}\bm{u}_k\rightharpoonup\mathbb{A}\bm{u}\quad\text{ in }L^2(0,T;\mathbf{L}^2(\Omega)),\\&(\bm{u}_k,\Sigma \bm{u}_k)\to(\bm{u},\Sigma \bm{u})\text{ a.e. on }]0,T[\times\Omega,\text{ and}\\&|\bm{u}_k|+|\Sigma \bm{u}_k|\leq U \quad\text{ a.e. on }]0,T[\times\Omega,\quad\forall k\in\mathbb{N},\end{split}\end{equation} for some $U\in L^2_{loc}(]0,T[\times\Omega)$.

By continuity of $\Pi$ and $\bm{f}$, since $\Omega_k$ is increasing to $\Omega$, \[\Pi(t,x,\bm{u}_k,\Sigma \bm{u}_k)\to\Pi(t,x,\bm{u},\Sigma \bm{u}),\quad\text{ and }\]\[\chi_{\Omega_k}(x)\bm{f}(t,x,\bm{u}_k,\Sigma \bm{u}_k)\to \bm{f}(t,x,\bm{u},\Sigma \bm{u})\text{ a.e. on }]0,T[\times\Omega.\] By the uniform boundedness of $\bm{f}_{\bm{u}_k,k}$ in \eqref{fdomination} and the domination of $\bm{u}_k$ by $U$ in \eqref{OmegakOmegaConvg}, we have \[|\chi_{\Omega_k}(x)\bm{f}(t,x,\bm{u}_k,\Sigma \bm{u}_k)|\leq F+(\Lambda_1+\Lambda_2) U \quad\text{ a.e. on }]0,T[\times\Omega,\quad\forall k\in\mathbb{N}.\] Also, as in \eqref{ConvgNonlin4}, the convergences in \eqref{OmegakOmegaConvg} imply that \begin{equation}\label{ConvgNonlin5}\Pi(t,x,\bm{u}_k,\Sigma \bm{u}_k)\mathbb{A}\bm{u}_k\rightharpoonup \Pi(t,x,\bm{u},\Sigma \bm{u})\mathbb{A}\bm{u}\quad\text{ weakly in }L^2(0,T;\mathbf{L}^2(\Omega)).\end{equation} 
Therefore, \[\chi_{\Omega_k}(x)\bm{f}(t,x,\bm{u}_k,\Sigma \bm{u}_k)=\bm{u}_k'+\Pi(t,x,\bm{u}_k,\Sigma \bm{u}_k)\mathbb{A}\bm{u}_k\text{ converges weakly in }L^2(0,T;\mathbf{L}^2(\Omega)).\] 
On the other hand, for every $\varphi\in L^2(0,T;\mathbf{C}_c(\Omega))$ compactly supported in $]0,T[\times\Omega$, we have \[\int_0^T\int_{\Omega}\chi_{\Omega_k}\bm{f}(t,x,\bm{u}_k,\Sigma \bm{u}_k)\cdot\varphi\to\int_0^T\int_\Omega \bm{f}(t,x,\bm{u},\Sigma \bm{u})\cdot\varphi\] by the dominated convergence theorem. 
Since compactly supported functions are dense in $L^2(0,T;\mathbf{L}^2(\Omega))$, we have the weak convergence \[\chi_{\Omega_k}\bm{f}(t,x,\bm{u}_k,\Sigma \bm{u}_k)\rightharpoonup \chi_{\Omega}\bm{f}(t,x,\bm{u},\Sigma \bm{u})= \bm{f}(t,x,\bm{u},\Sigma \bm{u})\quad\text{ weakly in }L^2(0,T;\mathbf{L}^2(\Omega)).\] Letting $k\to\infty$ in the problem \eqref{QuasilinNonautonomousProbOmegak}, we therefore obtain that $\bm{u}$ satisfies the original problem \[\bm{u}'+\Pi(t,x,\bm{u},\Sigma \bm{u})\mathbb{A}\bm{u}=\bm{f}(t,x,\bm{u},\Sigma \bm{u})\quad\text{ in }\Omega \text{ for a.e. }t\in]0,T[.\] Furthermore, invoking the continuity $MR\hookrightarrow C([0,T];\mathbf{H}^s_0(\Omega))$ by Lemma \ref{ProductRuleTime} as before, $\bm{u}_k(0)\rightharpoonup \bm{u}(0)$ in $\mathbf{H}^s_0(\Omega)$, so $\bm{u}(0)=\bm{u}_0$. Thus, $\bm{u}$ is a solution to the problem \eqref{EstQuasilinNonautCauchySupercritical}. Furthermore, since the estimate \eqref{EstNonlin2} is independent of $k$, we can pass to the limit to obtain the estimate \eqref{EstQuasilinNonautCauchySupercritical}.

\begin{remark}
It is also possible to consider a different nonlocal vectorial operator $\mathbb{A}\bm{u}=(\mathbb{A}_1u^1,\dots,\mathbb{A}_mu^m)$ for each equation in the system 
\begin{equation*}\bm{u}'+
\Pi(t,x,\bm{u},\Sigma \bm{u})\mathbb{A}\bm{u}=\bm{f}(t,x,\bm{u},\Sigma \bm{u})
\quad\text{ in }]0,T[\times\Omega,\end{equation*}
for $\mathbb{A}_i$ given by (possibly different) scalar operators satisfying \eqref{LEllipBdd}, which may be of the form \eqref{ClassicalDeriv}, \eqref{FracOp} or \eqref{NonlocalOp}, and $\Pi$ satisfy the same assumptions.
\end{remark}

\begin{remark}\label{InhomogDirichletBdryCond}
The results in Theorem \ref{NonlinNonautonomousProbThmSupercritical} can in fact be extended to the inhomogeneous Dirichlet boundary problem $\bm{u}=\bm{g}$ in $]0,T[\times\Omega^c$.

Indeed, writing $MR(\mathbb{R}^n)$ for \[MR(\mathbb{R}^n):=H^1(0,T;\mathbf{L}^2(\mathbb{R}^n))\cap \{\bm{u}\in \mathbf{H}^s(\mathbb{R}^n):\mathbb{A}\bm{u}\in \mathbf{L}^2(\mathbb{R}^n)\},\] let $\bm{g}\in MR(\mathbb{R}^n)\cap L^2(0,T;\mathbf{H}^{s+\theta}(\mathbb{R}^n))\cap C([0,T];\mathbf{H}^s(\mathbb{R}^n))$, such that $\bm{g}(0)\in \mathbf{H}^s(\mathbb{R}^n)$. Considering $\bar{\bm{u}}=\bm{u}-\bm{g}$, we can solve the problem for  $\bar{\bm{u}}\in MR(\Omega)$, for the corresponding translated problem.
\end{remark}

\subsection{Examples}

\subsubsection{Quasilinear System with the Classical Laplacian}\label{subsect:MRQuasiLinearAppClassical}

As a first example, we consider the classical Laplacian $\Delta$ in the case of $s=1$ as in Example 5.1 of \cite{ArendtChill}, extended to the case of a system of equations.

Considering the vectorial Laplacian $(-\Delta)_m$ defined by 
\[(-\Delta)_m=\begin{bmatrix}
-c_1\Delta &  & 0\\
 & \ddots & \\
0 &  & -c_m\Delta
\end{bmatrix}\]
for constants $c_1,\cdots,c_m>0$. Then, applying Theorem \ref{NonlinNonautonomousProbThmSupercritical} for $\mathbb{A}=(-\Delta)_m$, we have

\begin{corollary}
Suppose $\Pi$ and $\bm{f}$ satisfy the assumptions of Theorem \ref{NonlinNonautonomousProbThmSupercritical} with $\Omega$ being an open bounded Lipschitz domain. Then, writing for the gradient $\partial=(\partial_1,\dots,\partial_n)$, for every $\bm{u}_0\in\mathbf{H}^s_0(\Omega)$, the nonlinear problem given by 
\begin{equation*}\begin{cases}\bm{u}'(t)+\Pi(t,x,\bm{u},\partial \bm{u})(-\Delta)_m\bm{u}(t)=\bm{f}(t,x,\bm{u},\partial \bm{u})\quad\text{ for a.e. }t\in]0,T[,\\\bm{u}=\bm{0}\text{ a.e. on }]0,T[\times\partial\Omega,\quad\quad\bm{u}(0,\cdot)=\bm{u}_0(\cdot)\text{ a.e. in }\Omega\end{cases}\end{equation*} has a solution $\bm{u}\in H^1(0,T;\mathbf{L}^2(\Omega))\cap L^2(0,T;\mathbf{H}^2(\Omega))\cap C([0,T];\mathbf{H}^1_0(\Omega))$ for any $T\in]0,\infty[$. 
Moreover, there exists a constant $c'=c'(\underline{\gamma},\bar{\gamma},\Lambda_1,\Lambda_2,T)>0$ such that every solution $\bm{u}$ satisfies \begin{equation*}\norm{\bm{u}}_{MR}\leq c'\left(\norm{F}_{L^2(0,T;L^2(\Omega))}+\norm{\bm{u}_0}_{\mathbf{H}^1_0(\Omega)}\right).\end{equation*}

\end{corollary}

In particular, this extends the results of \cite{ArendtChill} to system of equations.

\subsubsection{Approximation Models for Interacting Nonlocal Diffusive Species Populations}\label{subsect:MRQuasiLinearAppPopulation}
We also consider a nonlocal version of cross-diffusive systems modelling two interacting species, 
given for $0<s\leq1$ by \begin{align}&\begin{aligned}u'=-D_1(u,v,\Sigma u, \Sigma v)(-\Delta)^su+R_1(u,v,\Sigma u, \Sigma v),\\ v'=-D_2(u,v,\Sigma u, \Sigma v)(-\Delta)^sv+R_2(u,v,\Sigma u, \Sigma v),\end{aligned}\quad &x\in \Omega,t>0\nonumber\\&u(0,x)=u_0(x),\quad v(0,x)=v_0(x),\quad &x\in\Omega,\label{CrossDiffEq}\\&u(t,x)=v(t,x)=0\quad &x\in\Omega^c,t>0,\nonumber
\end{align} where $\Sigma$ has order $\sigma$ with $0<\sigma\leq s<1$. The diffusion coefficients $D_1$ and $D_2$ are bounded and strictly positive, describing a controlled nonlocal nonlinear spreading of the biological population which is dependent both on the size and density of the species itself and the other species. The interaction between the two species is both in terms of space competition with regard to diffusion, as well as in the linear bounded reaction terms $R_1(u,v)$ and $R_2(u,v)$.
Here, the spreading can be represented by any of the operators in Section \ref{sect:MRQuasiLinear}, and can be both nonlocal, as described by the nonlocal operator $\mathbb{L}^s_A$, or local, given by the classical operator $\mathbb{L}$.
Such systems with constant diffusion coefficients may appear in activator-inhibitor systems with linear or sublinear kinetic functions (see, for instance, Chapter 9 of \cite{Yagi2010BookParabolicEvolEqs}).

This model can also be obtained as an approximation of Lotka–Volterra-type models, where the reaction terms are obtained from linearising quadratic terms describing predator-prey, competition or cooperation interactions. For instance, the Shigesada-Kawasaki-Teramoto system (see, for instance, \cite{LiZhao2005DCDSCrossDiffGlobal} and \cite{ChenJungel2006JDECrossDiff}) given by \[\bm{u}'+D(-\Delta)^s_2\bm{u}=R\quad x\in\Omega\] for $\bm{u}=(u_1,u_2)$ with diffusion matrix
\[D=\begin{bmatrix}
\begin{gathered}d_1+q\rho_{11}u_1^q+\rho_{12}u_2^q\\+q\rho_{13}(\mathcal{D}^su_1)^q+\rho_{14}(\mathcal{D}^su_2)^q\end{gathered} & q\rho_{12}u_1u_2^{q-1}\\ \\
q\rho_{21}u_1^{q-1}u_2 & \begin{gathered}d_2+q\rho_{22}u_2^q+\rho_{21}u_1^q\\+q\rho_{23}(\mathcal{D}^su_2)^q+\rho_{24}(\mathcal{D}^su_1)^q\end{gathered}
\end{bmatrix}\] and reaction term \[R=(R_1,R_2),\quad R_i=(a_{1,i}-b_{1,i}u_1-c_{1,i}u_2)u_i+(a_{2,i}-b_{2,i}\mathcal{D}^su_1-c_{2,i}\mathcal{D}^su_2)\mathcal{D}^su_i\]
can be approximated via the logistic function, which is a bounded nonlinearity, \[u_i\sim \frac{u_i}{1+\epsilon|u_i|}=:\tilde{u}_i\quad \epsilon>0\]\[\mathcal{D}^su_i\sim \frac{\mathcal{D}^su_i}{1+\epsilon|\mathcal{D}^su_i|}=:\widetilde{\mathcal{D}^su_i}\quad \epsilon>0\] to obtain the system \[\bm{u}'+D_{approx}(-\Delta)^s\bm{u}=R_{approx}\quad x\in\Omega\] where $D_{approx}$ is of the form 
\[D_{approx}=\begin{bmatrix}
\begin{gathered}d_1+q\rho_{11}\tilde{u}_1^q+\rho_{12}\tilde{u}_2^q\\+q\rho_{13}(\widetilde{\mathcal{D}^su_1})^q+\rho_{14}(\widetilde{\mathcal{D}^su_2})^q\end{gathered} & q\rho_{12}\tilde{u}_1\tilde{u}_2^{q-1}\\ \\
q\rho_{21}\tilde{u}_1^{q-1}\tilde{u}_2 & \begin{gathered}d_2+q\rho_{22}\tilde{u}_2^q+\rho_{21}\tilde{u}_1^q\\+q\rho_{23}(\widetilde{\mathcal{D}^su_2})^q+\rho_{24}(\widetilde{\mathcal{D}^su_1})^q\end{gathered}
\end{bmatrix}\] 
and 
\[R_{approx}=\begin{bmatrix}
(a_{1,1}-b_{1,1}\tilde{u}_1-c_{1,1}\tilde{u}_2)u_1+(a_{2,1}-b_{2,1}\widetilde{\mathcal{D}^su_1}-c_{2,1}\widetilde{\mathcal{D}^su_2})\mathcal{D}^su_1\\
(a_{1,2}-b_{1,2}\tilde{u}_1-c_{1,2}\tilde{u}_2)u_2+(a_{2,2}-b_{2,2}\widetilde{\mathcal{D}^su_1}-c_{2,2}\widetilde{\mathcal{D}^su_2})\mathcal{D}^su_2
\end{bmatrix}\] 
so that $D_{approx}$ is bounded, under appropriate assumptions on $\rho_{ij}$ and $\bm{u}$, and $R_{approx}$ has a linear growth on $u_1$ and $u_2$ for any fixed $\epsilon$, fulfilling our assumptions.

Supposing $D_{approx}$ satisfies \eqref{GammaNonDivCond}, which is obtained by taking the sufficient assumption that $u_1,u_2$ are positive and that the diffusion coefficients $\rho_{ij}>0$ are bounded, and the derivatives are of order $s$, 
by Theorem \ref{NonlinNonautonomousProbThmSupercritical}, this problem admits a global solution $\bm{u}$ in \[\bm{u}\in H^1(0,T;\mathbf{L}^2(\Omega))\cap L^2(0,T;\mathbf{H}^{s'}(\Omega))\cap C([0,T];\mathbf{H}^s_0(\Omega)),\] where $s'=\min\{2s,s+\frac12\}-\epsilon$ for any $\epsilon>0$. This means that we can consider a more general case of cross-diffusion involving nonlocal operators, and obtaining a regularity result that is comparable to the classical cases in Theorem B of \cite{LiZhao2005DCDSCrossDiffGlobal} and Theorem 1.3 of \cite{BendahmaneLanglais2010JEvolEqCrossDiffLinearizedEq1.10}.

\section{The Nonlinear Problem $s<\sigma<2s\leq2$ with $\Omega$ Bounded
}\label{sect:MRQuasiSigmaBig} 

In this section, we want to further extend the result to higher order derivatives $\sigma>s>0$. In particular, $\sigma$ may be greater than 1, generalising the scalar quasilinear diffusion equations in the classical case in \cite{ArendtChill}. Here, we focus on the classical elliptic operator $\mathbb{L}$ as given in Example 1 of Section \ref{sect:MRQuasiLinear} defined by \eqref{ClassicalDeriv}, as well as the nonlocal fractional Laplacian defined in Example 2 of Section \ref{sect:MRQuasiLinear} defined by \eqref{FracLap}, since we have additional regularity results for those cases. Then, by the results of \cite{DongKim2011ARMALpParabolicEllipticSystemsBMO}, and \cite{GrubbHormander} and \cite{borthagaray2021besov}, we know that there exists a unique solution to the Dirichlet problem associated with $\mathbb{L}$ and with $(-\Delta)^s$, given by Propositions \ref{ClassicalEllipRegThmVectorial} and \ref{BNochettoFracLapReg} respectively. Therefore, the spaces $\mathbf{L}^2_{\mathbb{L}}$ and $\mathbf{L}^2_{(-\Delta)^s_m}$ make sense. Furthermore, it is clear that $\mathbb{L}$ and $(-\Delta)^s_m$ are bounded and $\mathbf{L}^2(\Omega)$-coercive. 

We first recall the following Poincar\'e inequality concerning the embedding of $\mathbf{L}^2(\Omega)$ in $\mathbf{H}^s_0(\Omega)$. See, for instance, Theorem 2.2 of \cite{BellidoCuetoMoraCorral2020PiolaVector}.
\begin{lemma}[Poincar\'e inequality]\label{Poincare}
Let $s\in]0,1]$. Then for any open bounded set $\Omega\subset\mathbb{R}^n$, there exists a constant $c_P>0$ depending only on $\Omega$, $n$ and $s$ such that \[c_P\norm{\bm{u}}_{\mathbf{L}^2(\Omega)}\leq\norm{D^s\bm{u}}_{L^2(\mathbb{R}^n)}^2.\] for all $\bm{u}\in \mathbf{H}^s_0(\Omega)$. In particular, we have the equivalence of the norms $\norm{\cdot}_{\mathbf{L}^2(\Omega)}$ and $\norm{\cdot}_{\mathbf{H}^s_0(\Omega)}$.
\end{lemma}

Assume $\mathbb{A}$ satisfies $\mathbf{L}^2_{\mathbb{A}}:=\{\bm{u}\in \mathbf{H}^s_0(\Omega):\mathbb{A}\bm{u}\in \mathbf{L}^2(\Omega)\}\subset \mathbf{H}^{\sigma'}(\Omega)$ for some $s<\sigma'<2s$ for $\Omega$ bounded and Lipschitz domain, i.e. there exists a constant constant $C_\mathbb{A}>0$ and $\mu'\geq0$ such that \begin{equation}\label{OpRegBound}\norm{\bm{u}}_{L^2(0,T;\mathbf{H}^{\sigma'}(\Omega))}\leq C_\mathbb{A}\norm{\mathbb{A}\bm{u}}_{L^2(0,T;\mathbf{L}^2(\Omega))}+\mu'\norm{\bm{u}}_{L^2(0,T;\mathbf{L}^2(\Omega))}.\end{equation} 
In particular, $\sigma'=2$ for $\mathbb{A}=\mathbb{L}$, and $\sigma'=\min\{2s,s+\frac12\}$ for $\mathbb{A}=(-\Delta)^s_m$ where $0<s<1$. Therefore, applying the compact embedding \eqref{CompactEmbeddingGeneral}, we obtain that \begin{equation}\label{MRECptEmbedBigSigma}MR=H^1(0,T;\mathbf{L}^2(\Omega))\cap L^2(0,T;\mathbf{L}^2_{\mathbb{A}})\hookrightarrow L^2(0,T;\mathbf{H}^\sigma(\Omega))=\tilde{E}\end{equation} is compact for any open Lipschitz bounded set $\Omega\subseteq\mathbb{R}^n$, for any $\sigma<\sigma'$.

Also, by the Sobolev embeddings, there exists a Sobolev constant $0<c_S<1$ depending on $s<\sigma'<2s$, $\sigma$ and $\Omega$, 
such that \begin{equation}\label{SobolevCompareNorm}c_S\norm{\bm{v}}_{L^2(0,T;\mathbf{H}^{\sigma}(\Omega))}^2\leq\norm{\bm{v}}_{L^2(0,T;\mathbf{H}^{\sigma'}(\Omega))}^2\quad\forall\bm{v}\in L^2(0,T;\mathbf{H}^{\sigma'}(\Omega)).\end{equation}

Then, assuming  $\bm{f}:]0,T[\times\Omega\times\mathbb{R}^m\times\mathbb{R}^q\to\mathbb{R}^m$ satisfy the assumptions of Theorem \ref{NonlinNonautonomousProbThmSupercritical} such that 
\begin{equation}\label{fcondNonlinBigSigmaSublinGrowth}|\bm{f}(t,x,\bm{u},\bm{p})|\leq F(t,x)+\Lambda_1|\bm{u}|+\Lambda_2|\bm{p}|^\alpha \end{equation} for some $F\in L^2(0,T;L^2(\Omega))$, $\Lambda_1,\Lambda_2\geq0$, such that either
\begin{enumerate}[label=(\roman*)]
    \item $0<\alpha<1$, or
    \item $\alpha=1$ with 
\begin{equation}\label{fcondNonlinBigSigma}0<\Lambda_2\leq \underline{\gamma}\sqrt{\frac{c_S}{C_\mathbb{A}}},\end{equation}
\end{enumerate}
we have the following result:

\begin{theorem}\label{NonlinNonautonomousProbThmBigSigma} 
Suppose $\Omega\subset\mathbb{R}^n$ is a Lipschitz bounded open set. 
Let $\Pi:]0,T[\times\Omega\times\mathbb{R}^m\times\mathbb{R}^q\to\mathbb{R}^{m\times m}$ satisfy the assumptions of Theorem \ref{NonlinNonautonomousProbThmSupercritical}, and $\bm{f}:]0,T[\times\Omega\times\mathbb{R}^m\times\mathbb{R}^q\to\mathbb{R}^m$ satisfy the assumptions \eqref{fcondNonlinBigSigmaSublinGrowth} above for either condition (i) or (ii).
Suppose $\mathbb{A}$ satisfies \eqref{OpRegBound} for some $s<\sigma'<2s\leq2$.
Then, for any $\sigma<\sigma'$ and every $\bm{u}_0$ such that $\bm{u}_0\in \mathbf{H}^s_0(\Omega)\cap\mathbf{H}^\sigma(\Omega)$, there exists \begin{equation}\label{QuasilinNonautonomousBigSigmaSolnReg}\bm{u}\in H^1(0,T;\mathbf{L}^2(\Omega))\cap L^2(0,T;\mathbf{L}^2_{\mathbb{A}})\cap L^2(0,T;\mathbf{H}^\sigma(\Omega))\cap C([0,T];\mathbf{H}^s_0(\Omega))\end{equation} solving the problem \begin{equation}\label{QuasilinNonautonomousProbBigSigma}\begin{split}& \bm{u}'(t)+\Pi(t,x,\bm{u},\Sigma \bm{u})\mathbb{A}\bm{u}(t)=\bm{f}(t,x,\bm{u},\Sigma \bm{u})\quad\text{ for a.e. }t\in]0,T[,\\&\bm{u}(0)=\bm{u}_0,\end{split}\end{equation} where $\Sigma$ represents fractional derivatives of order $\sigma$ which may be greater than 1.  Moreover, there exists a constant $c''=c''(\Omega,\underline{\gamma},\bar{\gamma},a_*,a^*,C_\mathbb{A},\Lambda_1,\Lambda_2,T,\alpha)>0$ such that for every solution $\bm{u}$ of \eqref{QuasilinNonautonomousProbBigSigma}, \begin{equation}\label{EstQuasilinNonautCauchyBigSigma}\norm{\bm{u}}_{MR}\leq c''\left(\norm{F}_{L^2(0,T;L^2(\Omega))}+\norm{\bm{u}_0}_{\mathbf{H}^\sigma_0(\Omega)}\right).\end{equation}
\end{theorem}

\begin{proof}
Most of the proof follows the argument of Theorem \ref{NonlinNonautonomousProbThmSupercritical}, this time applying the Leray-Schauder fixed point theorem for the fixed point constructed in the Banach space $\tilde{E}$ in \eqref{MRECptEmbedBigSigma}, for $\sigma<\sigma'$, where $MR$ is compactly embedded. In particular, this means that we do not have to consider the sequence of sets $\Omega_k$, and we can directly consider the compact map $\mathcal{T}$ defined by $\bm{u}=:\mathcal{T}\bm{v}\in MR$ of the problem \begin{align*}& \bm{u}'(t)+\Pi_{\bm{v}}(t,\cdot)\mathbb{A}\bm{u}(t)=\bm{f}_{\bm{v}}(t,\cdot)\quad\text{ for a.e. }t\in]0,T[,\text{ and }\\&\bm{u}(0)=\bm{u}_0\nonumber\end{align*} 

A major modification lies in the proof that the Leray-Schauder set \begin{equation}\label{SchauderSetBigSigma}\mathscr{S}=\{\bm{u}\in \tilde{E}:\bm{u}=\lambda \mathcal{T}\bm{u}\text{ for some }\lambda\in[0,1]\}\end{equation} is bounded.  In particular, the proof of the a priori estimate in \eqref{EstNonlin2} needs to be modified for the case of $\sigma\geq s$.

Indeed, we obtain the bound on $\norm{\bm{u}}_{L^2(0,T;\mathbf{H}^\sigma_0(\Omega))}$ for $\sigma\geq s$ as follows: Multiplying the equation \eqref{ApproxEqOmegak} by $\mathbb{A}\bm{u}$ and integrating over $\Omega$, we obtain, by the bounds \eqref{GammaNonDivCond} and making use of Lemma \ref{ProductRuleTime} and the Cauchy-Schwarz inequality, for a.e. $t>0$, \begin{multline*}\frac{1}{2}\frac{d}{dt}  \int_\Omega \bm{u}\cdot\mathbb{A}\bm{u}+\underline{\gamma}\norm{\mathbb{A}\bm{u}}_{\mathbf{L}^2(\Omega)}^2\\\leq \int_\Omega  \bm{u}'\cdot \mathbb{A}\bm{u}+\int_\Omega \Pi_u\mathbb{A}\bm{u}\cdot\mathbb{A}\bm{u}=\int_{\Omega} \bm{f}\cdot \mathbb{A}\bm{u}\\\leq\frac{1}{2\underline{\gamma}}\norm{\bm{f}}_{\mathbf{L}^2(\Omega)}^2+\frac{\underline{\gamma}}{2}\norm{\mathbb{A}\bm{u}}_{\mathbf{L}^2(\mathbb{R}^n)}^2.\end{multline*}
Integrating over time on $]0,t[$ for any $t\leq T$, it follows by \eqref{LEllipBdd} that \begin{equation*}  a_*\norm{\bm{u}(t)}_{\mathbf{H}^s_0(\Omega)}^2+\underline{\gamma}\int_0^t\norm{\mathbb{A}\bm{u}}_{\mathbf{L}^2(\Omega)}^2\leq a^*\norm{\bm{u}_0}_{\mathbf{H}^s_0(\Omega)}^2+\mu\norm{\bm{u}(t)}_{\mathbf{L}^2(\Omega)}^2+\frac{1}{\underline{\gamma}}\int_0^t\norm{\bm{f}}_{\mathbf{L}^2(\Omega)}^2,\end{equation*} and so, taking the supremum over $t\in]0,T[$ and 
making use of \eqref{L2HsnormIneq}, we have \begin{align*}  a_*\norm{\bm{u}}_{L^\infty(0,T;\mathbf{H}^s_0(\Omega))}^2+\underline{\gamma}\norm{\mathbb{A}\bm{u}}_{L^2(0,T;\mathbf{L}^2(\Omega))}^2
&\leq a^*\norm{\bm{u}_0}_{\mathbf{H}^s_0(\Omega)}^2+\mu\norm{\bm{u}_0}_{\mathbf{L}^2(\Omega)}^2+2\mu^2\bar{\gamma}\norm{\bm{u}}_{L^2(0,T;\mathbf{L}^2(\Omega))}^2\\&\quad+\frac{1}{2\bar{\gamma}}\norm{ \bm{u}'}_{L^2(0,T;\mathbf{L}^2(\Omega))}^2+\frac{1}{\underline{\gamma}}\norm{\bm{f}}_{L^2(0,T;\mathbf{L}^2(\Omega))}^2.\end{align*}  
Considering only the term $\norm{\mathbb{A}\bm{u}}_{L^2(0,T;\mathbf{L}^2(\Omega))}^2$ on the left-hand-side of the inequality and applying Assumptions \eqref{OpRegBound} and \eqref{fcondNonlin} then gives \begin{align*}  \frac{\underline{\gamma}}{C_\mathbb{A}}\norm{\bm{u}}_{L^2(0,T;\mathbf{H}^{\sigma'}(\Omega))}^2&\leq\frac{\mu'}{C_\mathbb{A}}\norm{\bm{u}}_{L^2(0,T;\mathbf{L}^2(\Omega))}^2 + a^*\norm{\bm{u}_0}_{\mathbf{H}^s_0(\Omega)}^2+\mu\norm{\bm{u}_0}_{\mathbf{L}^2(\Omega)}^2\\&\quad+2\mu^2\bar{\gamma}\norm{\bm{u}}_{L^2(0,T;\mathbf{L}^2(\Omega))}^2+\frac{1}{2\bar{\gamma}}\norm{ \bm{u}'}_{L^2(0,T;\mathbf{L}^2(\Omega))}^2+\frac{1}{\underline{\gamma}}\norm{\bm{f}}_{L^2(0,T;\mathbf{L}^2(\Omega))}^2\\&\leq a^*\norm{\bm{u}_0}_{\mathbf{H}^s_0(\Omega)}^2+\mu\norm{\bm{u}_0}_{\mathbf{L}^2(\Omega)}^2+\left(\frac{\mu'}{C_\mathbb{A}}+2\mu^2\bar{\gamma}+\frac{\Lambda_1^2}{\underline{\gamma}}\right)\norm{\bm{u}}_{L^2(0,T;\mathbf{L}^2(\Omega))}^2\\&\quad+\frac{1}{2\bar{\gamma}}\norm{ \bm{u}'}_{L^2(0,T;\mathbf{L}^2(\Omega))}^2+\frac{1}{\underline{\gamma}}\norm{F}_{L^2(0,T;L^2(\Omega))}^2+\frac{\Lambda_2^2}{\underline{\gamma}}\int_0^T\norm{D^\sigma\bm{u}}_{\mathbf{L}^2(\mathbb{R}^n)}^{2\alpha}.\end{align*} Next, we argue as in estimate \eqref{EstNonlin2} to control the $H^1(0,T;\mathbf{L}^2(\Omega))$-norm of $\bm{u}$, and there exists a constant $c_5=c_5(\underline{\gamma},\bar{\gamma},a_*,a^*,\mu,\Lambda_1,\Lambda_2,T)>0$ such that \begin{equation*}  \frac{\underline{\gamma}}{C_\mathbb{A}}\norm{\bm{u}}_{L^2(0,T;\mathbf{H}^{\sigma'}(\Omega))}^2\leq c_5\left(\norm{\bm{u}_0}_{\mathbf{H}^s_0(\Omega)}^2+\norm{F}_{L^2(0,T;L^2(\Omega))}^2\right)+\frac{\Lambda_2^2}{\underline{\gamma}}\int_0^T\norm{D^\sigma\bm{u}}_{\mathbf{L}^2(\mathbb{R}^n)}^{2\alpha}.\end{equation*} 
Therefore, by \eqref{SobolevCompareNorm}, 
\begin{equation*}  \frac{c_S\underline{\gamma}}{C_\mathbb{A}}\norm{\bm{u}}_{L^2(0,T;\mathbf{H}^{\sigma}(\Omega))}^2\leq c_5\left(\norm{\bm{u}_0}_{\mathbf{H}^s_0(\Omega)}^2+\norm{F}_{L^2(0,T;L^2(\Omega))}^2\right)+\frac{\Lambda_2^2}{\underline{\gamma}}\norm{\bm{u}}_{L^2(0,T;\mathbf{H}^\sigma(\Omega))}^{2\alpha}.\end{equation*}
By the Assumption \eqref{fcondNonlinBigSigmaSublinGrowth} with either Condition (i) with $\alpha<1$ or Condition (ii) with $\alpha=1$ and $\Lambda_2<\underline{\gamma}\sqrt{\frac{c_S}{C_\mathbb{A}}}$, we obtain an a priori bound on the term $\norm{\bm{u}}_{L^2(0,T;\mathbf{H}^\sigma(\Omega))}$.

Also, as in the estimate \eqref{EstNonlin2}, the Leray-Schauder set $\mathscr{S}$ given by \eqref{SchauderSetBigSigma} is bounded in $MR$. Making use of the Aubin-Lions compactness lemma $MR\hookrightarrow\tilde{E}$, by the Leray-Schauder principle, $\mathcal{T}$ has a fixed point $\bm{u}$ satisfying \eqref{QuasilinNonautonomousBigSigmaSolnReg} and \eqref{EstQuasilinNonautCauchyBigSigma} solving the problem \eqref{QuasilinNonautonomousProbBigSigma}.
\end{proof}

\begin{remark}
The theorem holds with $\sigma'=2$ for $\mathbb{A}=\mathbb{L}$ as well as with $\sigma'=\min\{2s,s+\frac12\}$ for $\mathbb{A}=(-\Delta)^s_m$ for $0<s<1$, with derivatives of order $s<\sigma<\sigma'$, which may possibly be of order greater than 1. However, while derivatives of order less than 1 may take the form of $D^s$ or $\mathcal{D}^s$ as defined by \eqref{FracDervs} and \eqref{NonlocalDervs} respectively, the derivatives of order greater than 1 is only defined by \eqref{HigherOrderFracDervs}, and $\mathcal{D}^s\bm{u}$ is not defined for $s>1$. The derivatives of order equal to 1 is just the classical gradient.

\end{remark}

\begin{remark}
As in Remark \ref{InhomogDirichletBdryCond}, the results in Theorem \ref{NonlinNonautonomousProbThmBigSigma} can also be extended to the inhomogeneous Dirichlet boundary problem $\bm{u}=\bm{g}$ in $]0,T[\times\Omega^c$, for $\bm{g}\in MR(\mathbb{R}^n)\cap L^2(0,T;\mathbf{H}^{s+\theta}(\mathbb{R}^n))\cap C([0,T];\mathbf{H}^s(\mathbb{R}^n))$ such that $\bm{g}(0)\in \mathbf{H}^s(\mathbb{R}^n)$.
\end{remark}

As a result, we can consider quasilinear diffusion equations and systems with derivatives of order $\sigma>s$ such that $\sigma$ may be greater than 1, generalising the results of \cite{ArendtChill}, \cite{ArendtDierSystems} and \cite{LaasriMugnolo2020MMAS}. This provides many useful applications, particularly in advection-diffusion systems, as seen in Section \ref{subsect:MRQuasiLinearAppSQG}.

\subsection{Examples}

\subsubsection{A System with the Classical Laplacian with $D^\sigma$-quasilinearity $1<\sigma<2$}

As a first application, we take a relook at the vectorial classical Laplacian $(-\Delta)_m$ in Section \ref{subsect:MRQuasiLinearAppClassical}, this time, with the $D^\sigma$ fractional derivatives for any $1<\sigma<2$ as defined in \eqref{HigherOrderFracDervs}.

\begin{corollary}
Suppose $\Pi$ and $\bm{f}$ satisfy the assumptions of Theorem \ref{NonlinNonautonomousProbThmBigSigma} with $\bm{f}$ fulfilling either Conditions (i) or (ii) of \eqref{fcondNonlinBigSigmaSublinGrowth}. Then, for $1<\sigma<2$ and every $\bm{u}_0\in\mathbf{H}^1_0(\Omega)\cap\mathbf{H}^\sigma(\Omega)$, the nonlinear problem given by 
\begin{equation*}\bm{u}'(t)+\Pi(t,x,\bm{u},D^\sigma \bm{u})(-\Delta)_m\bm{u}(t)=\bm{f}(t,x,\bm{u},D^\sigma \bm{u})\quad\text{ for a.e. }t\in]0,T[,\end{equation*} has a solution $\bm{u}\in H^1(0,T;\mathbf{L}^2(\Omega))\cap L^2(0,T;\mathbf{H}^2(\Omega))\cap C([0,T];\mathbf{H}^1_0(\Omega))$ satisfying \eqref{EstQuasilinNonautCauchyBigSigma} for any $T\in]0,\infty[$. 

\end{corollary}

In particular, this further extends the results of \cite{ArendtChill} to include the fractional derivatives $D^\sigma$ of order $1<\sigma<2$.

\subsubsection{Anisotropic Advection-Diffusion Fractional Equations for $s>\frac12$}\label{subsect:MRQuasiLinearAppSQG}
Our last application is a semilinear anisotropic advection-diffusion system of equations. Such a system may be useful to transport models with fractional diffusion and is inspired, in the scalar case, by the 2-dimensional forced subcritical surface quasi-geostrophic flows with nonlocal dissipation (see, for instance,   \cite{ConstantinNguyen2018PhysDGlobalSQGBoundedDomains}) and the 2-dimensional Navier-Stokes equation.

Suppose $s>\frac{1}{2}$. Let $\bm{v}(t,x)$ be a bounded velocity field in $]0,T[\times\Omega$ in a bounded $\Omega\subset\mathbb{R}^n$ such that \begin{equation}\label{MRSQGBdd}\norm{\bm{v}}_{L^\infty(]0,T[\times\Omega)}\leq C_\#<\infty,\quad C_\#\text{ depending on }\Omega,\underline{\gamma},s \text{ and }\mathbb{A}\text{ as in \eqref{fcondNonlinBigSigma}}.\end{equation}  For $\bm{f}\in L^2(0,T;\mathbf{L}^2(\Omega))$ and $\bm{u}_0\in \mathbf{H}^s_0(\Omega)\cap\mathbf{H}^1(\Omega)$, the equation is given by \[\bm{u}'(t,x)+\Pi\mathbb{A}\bm{u}(t,x)=-\sum_{\alpha=1}^n v^\alpha(t,x)\partial_\alpha \bm{u}(t,x)+\bm{f}(t,x,\bm{u}),\quad (t,x)\in]0,T[\times\Omega\]\[\bm{u}(t,x)=\bm{0},\quad (t,x)\in]0,T[\times\Omega^c,\]\[\bm{u}(0,x)=\bm{u}_0(x),\quad x\in\Omega,\] where $\mathbb{A}=(-\Delta)^s_m$ or $\mathbb{L}$. 
Observe that this means that since $\frac12<s<1$, we have a convective term given by the classical gradient of $\bm{u}$. 

Since $\bm{v}$ is bounded as in \eqref{MRSQGBdd}, we can apply Theorem \ref{NonlinNonautonomousProbThmBigSigma} with $\sigma=1$ and with the source function given by the term $-\sum_\alpha v^\alpha(t,x)\partial_\alpha \bm{u}(t,x)+\bm{f}(t,x,\bm{u})$, such that \eqref{fcondNonlinBigSigmaSublinGrowth} is satisfied with $\alpha=1$. As a result, the problem admits a global solution \[\bm{u}\in H^1(0,T;\mathbf{L}^2(\Omega))\cap L^2(0,T;\mathbf{H}^1(\Omega))\cap C([0,T];\mathbf{H}^s_0(\Omega))\] for $0<s\leq1$ with $1=\sigma<\sigma'<2s=2$ for $\mathbb{A}=\mathbb{L}$ and $1=\sigma<\sigma'<\min\{2s,s+\frac12\}$ for $\mathbb{A}=(-\Delta)^s_m$. 

Furthermore, we do not require that $\bm{v}$ is divergence-free, which means that our result applies to compressible fluids as well. Such a result is new, as far as we know, since $\bm{v}$ is different from those considered in other works such as \cite{WangWuLiChen2014DCDSCompressibleSQG} and \cite{Debbi2020StochasticSQG}. However, by  \eqref{MRSQGBdd}, $\bm{v}$ must be bounded, which is a severe restriction, and therefore, in general, it may not cover the subcritical quasi-geostrophic model where $\bm{v}$ is given by the vorticity function of the Riesz transform of $\bm{u}$.

Moreover, limited by the elliptic regularity of $(-\Delta)^s$ in Proposition \ref{BNochettoFracLapReg}, we are only able to consider the subcritical $s>\frac{1}{2}$ case, and unable to obtain the critical $s=\frac{1}{2}$ nor the supercritical $s<\frac{1}{2}$ cases.
\\

\appendix
\appendix

\noindent \textbf{Acknowledgements.} 
C. Lo acknowledges the FCT PhD fellowship in the framework of the LisMath doctoral programme at the University of Lisbon. The research of J. F. Rodrigues was partially done under the framework of the Project PTDC/MATPUR/28686/2017 at CMAFcIO/ULisboa.

\printbibliography

\end{document}